\documentclass[12pt]{article}

\usepackage{amsmath,
amsfonts,latexsym,graphicx,amssymb,amsthm}
\usepackage{csquotes}
\usepackage{cjhebrew}
\usepackage{upgreek}
\usepackage{lipsum}

\usepackage{color}
\usepackage{authblk, url}
\usepackage[new]{old-arrows}

\newcommand{\na}{\mathbb{A}}
\newcommand{\C}{\mathbb{C}}
\newcommand{\F}{\mathbb{F}}

\newcommand{\HP}{\mathbb{H}}

\renewcommand{\P}{\mathbb{P}}
\newcommand{\Q}{\mathbb{Q}}
\newcommand{\R}{\mathbb{R}}
\newcommand{\Z}{\mathbb{Z}}

\newcommand{\osp}{\mathfrak{osp}}
\newcommand{\fsl}{\mathfrak{sl}}

\newcommand{\Aff}{\mathrm{Aff}}
\newcommand{\Aut}{\mathrm{Aut}}
\newcommand{\SAut}{\mathrm{SAut}}
\newcommand{\TAut}{\mathrm{TAut}}
\newcommand{\Card}{\mathrm{Card\,}}

\newcommand{\Core}{\mathrm{Core}}

\newcommand{\Elem}{\mathrm{Elem}}
\newcommand{\SElem}{\mathrm{SElem}}

\newcommand{\End}{\mathrm{End}}

\newcommand{\Ext}{\mathrm{Ext}}

\newcommand{\Gal}{{\mathrm {Gal}}}
\newcommand{\GL}{{\mathrm {GL}}}

\newcommand{\hdc}{{\mathrm{hdc}}}

\newcommand{\id}{\mathrm{id}}

\newcommand{\Image}{\mathrm{Im}}

\newcommand{\jac}{\mathrm{Jac}}
\newcommand{\Ker}{\mathrm{Ker}}
\newcommand{\lcm}{\mathrm{lcm}}

\newcommand{\Max}{\mathrm{Max}}

\newcommand{\New}{\mathrm{Newton}}

\newcommand{\PSL}{{\mathrm {PSL}}}

\newcommand{\sh}{{\mathrm {sh}}}
\newcommand{\SL}{{\mathrm {SL}}}
\newcommand{\SO}{{\mathrm {SO}}}
\newcommand{\St}{{\mathrm {St}}}

\newcommand{\rk}{\mathrm{rk}}

\newcommand{\Spec}{{\mathrm {Spec}}\,}

\newcommand{\Tr}{{\mathrm {Tr}}}
\newcommand{\trdeg}{{\mathrm {trdeg}}\,}
\newcommand{\Sup}{{\mathrm {Sup}}\,}
\newcommand{\Supp}{{\mathrm {Supp}}\,}

\renewcommand{\d}{\mathrm{d}}
\renewcommand{\mod}{\,\mathrm{mod}\,}

\newtheorem*{NLcrit}{Nonlinearity Criterion}
\newtheorem{Lcrit}{Linearity Criterion}

\newtheorem*{MainA}{Theorem A}
\newtheorem*{MainA1}{Theorem A.1}
\newtheorem*{MainA2}{Theorem A.2}
\newtheorem*{MainB}{Theorem B}

\newtheorem*{MainC}{Theorem C}
\newtheorem*{MainC1}{Theorem C.1}
\newtheorem*{MainC2}{Theorem C.2}
\newtheorem*{MainD}{Theorem D}

\newtheorem*{DSLemma}{Drutu-Sapir's Lemma}

\newtheorem{Cor}{Corollary}
\newtheorem*{Cor*}{Corollary}

\newtheorem*{Cor1.1}{Corollary 1.1}
\newtheorem*{Cor1.2}{Corollary 1.2}

\newtheorem*{vdKthm}{van der Kulk's Theorem}
\newtheorem*{Cornu}{Cornulier's Theorem}

\newtheorem*{ExA}{Example A}
\newtheorem*{ExB}{Example B}
\newtheorem*{ExC}{Example C}

\newtheorem*{Nagaothm}{Nagao's Theorem \cite{N}}

\newtheorem{lemma}{Lemma}

\newcommand{\ch}{\operatorname{ch}}

\font\small=cmr10

\MakeOuterQuote{"}

\AtEndDocument{
{\noindent\small  Institut  Camille Jordan du CNRS,
UCBL, 69622 Villeurbanne Cedex, FRANCE and

\noindent Shenzhen International Center for Mathematics, SUSTech, Shenzhen,  CHINA. 

\noindent Email: mathieu@math.univ-lyon1.fr} }

\title{Linearity and Nonlinearity of Groups of Polynomial Automorphisms of the Plane}

\author{Olivier Mathieu
\footnote{
Research supported by the UMR 5208 du CNRS
and the International Center for Mathematics at SUSTech}}

\begin{document}

\maketitle

\begin{abstract} Given a field $K$,
we investigate which subgroups
of the  group $\Aut\, \na^2_K$  of polynomial automorphisms of the plane are linear or not.

The results are contrasted. The group
$\Aut\, \na^2_K$ itself is nonlinear, except if 
$K$ is finite, but it contains some large subgroups,
of "codimension-five" or more,  which are linear.  This
phenomenon  is specific to dimension two: it is easy to prove that any natural "finite-codimensional" subgroup  of $\Aut\,\na^3_K$ is nonlinear, even for a finite field  $K$.

When $\ch\,K=0$, 
we also look at a similar questions for 
f.g. subgroups, and the results are again
disparate. The group $\Aut\, \na^2_K$ has a one-related f.g. subgroup which is not linear. 
However, there is a  large
subgroup, of "codimension-three", which is locally linear but not linear.
\end{abstract}

\noindent
\centerline{\it This paper is respectfully dedicated to the memory of Jacques Tits.}

\vfill
\eject

{\it Summary} 

Introduction

1. Main Definitions and Conventions

2. Mixed Products

3. Linearity over Rings vs. over Fields

4. A Nonlinear f.g. Subgroup of $\Aut_0\,\na_\Q^2$

5. The Linear Representation of $\Aut_{1}\,\na_K^2$

6. The Linear Representation of $\SAut_{0}^{<n}\,\na_K^2$

7. Semi-algebraic Characters

8. A Nonlinearity Criterion for $\Aut_S\,\na_K^2$

9. Two Linearity Criteria for $\Aut_S\,\na_K^2$

10. Examples of Linear or Nonlinear $\Aut_S\,\na_K^2$

11. Nonlinearity of Finite-Codimensional Subgroups of $\Aut\,\na_K^3$

\section*{Introduction}

Let $K$ be a field given once and for all, and
let $\Aut\,\na^2_K$ be the group of polynomial automorphisms  of the affine plane 
$\na^2_K$ over $K$.

\smallskip
\noindent
{\it 0.1 General Introduction} 

\noindent  A group $\Gamma$ is called {\it linear over a ring}, 
respectively {\it linear over a field}, if
there is an embedding $\Gamma\subset \GL(n,R)$,
resp. $\Gamma\subset \GL(n,L)$, for some positive integer $n$ and some commutative ring $R$,
resp. some field $L$.

Various authors have shown that the automorphism groups of algebraic varieties share many properties with the linear groups, e.g.
the Tits alternative holds in
$\Aut\,\na^2_\C$ \cite{L}, see also 
\cite{BL}\cite{S00}\cite {BPZ}. However, these groups are not always linear \cite{FP}\cite{P}.
In this paper, we will investigate the following
related 
 
\centerline{\it Question: which subgroups  of 
$\Aut\,\na^2_K$ are indeed linear or not?}

\noindent 
{\it Answer:} 
roughly speaking, $\Aut\,\na^2_K$ contains large linear subgroups and  small ones which are not.

In order to be more specific, let us  consider the following subgroups

\centerline{$\Aut_0\,\na^2_K=\{\phi\in \Aut\,\na^2_K\mid \phi({\bf 0})={\bf 0})\}$, }

\centerline{$\SAut_0\,\na^2_K=\{\phi\in \Aut_0\,\na^2_K\mid\jac(\phi)=1\}$, and}

\centerline{$\Aut_S\,\na^2_K=\{\phi\in \Aut_0\,\na^2_K\mid {\textnormal d}
\phi\vert_{\bf 0}\in S\}$, }

\noindent 
where $\jac(\phi):=\det\,\d\phi$ is the jacobian of $\phi$ and $S$ is a subgroup of $\GL(2,K)$.
Since $\Aut\,\na^2_K/\Aut_0\,\na^2_K$ is naturally isomorphic to $\na^2_K$, {\it informally speaking}
$\Aut_0\,\na^2_K$ is a subgroup of codimension two. Similarly, since $\jac(\phi)$ is  a constant polynomial, $\SAut_0\,\na^2_K$ has codimension three.
Let $U(K)$ be the group of linear transformations $(x,y)\mapsto (x,y+ax)$ for some $a\in K$.
Similarly, 
the group $\Aut_{U(K)}\,\na^2_K$ has codimension five,
and the  group 
$\Aut_1\,\na^2_K:=\Aut_{\{1\}}\,\na^2_K$ has codimension six.
Anyhow, they are viewed as large subgroups.

It was  known that the  Cremona group $Cr_2(\C)$  and $\Aut\,\na^2_\C$
are not linear over a field,
see \cite{C}\cite{Co}. For the large subgroups of
$\Aut\,\na^2_K$, the linearity results are  
 much more contrasted, as it is shown by the following

\begin{MainA} 
(A.1)  Whenever $K$ is infinite, the group 
$\SAut_0\,\na^2_K$ is not linear, even over a ring.

(A.2) However, there is an embedding
$\Aut_{U(K)}\,\na^2_K\subset \SL(2,K(t))$.
\end{MainA}
 
 \noindent 
For a finite field $K$, the index
$[\Aut\,\na^2_K:\Aut_{U(K)} \,\na^2_K]$ is finite. Therefore Theorem A.1 admits the following converse

\begin{Cor*} If $K$ is finite, the group 
$\Aut\,\na^2_K$ is linear over the field $K(t)$.
\end{Cor*}

The existence of
large linear subgroups 
in $\Aut\,\na^n_K$ is specific to the dimension
$2$. The case $n=3$ is enough to show this.
For a finite-codimensional ideal ${\bf m}$ of
$K[x,y,z]$, let $\Aut_{\bf m}\,\na^3_K$ be the group of
all automorphisms

\centerline{ $(x,y,z)\mapsto 
(x+f,y+g,z+h)$,}

\noindent where $f,\,g$ and $h$ belong to ${\bf m}$. Equivalently, $\phi$ fixes some
infinitesimal neighborhood of a finite subset in $\na^3_K$.  E.g. for
${\bf n}=(x,y,z)^2$, we have

\centerline
{$\Aut_{{\bf n}}\,\na^3_K=\{\phi\in \Aut\,\na^3_K \mid \phi({\bf 0})={\bf 0}\,$ and $\,
{\textnormal d}
\phi\vert_{\bf 0}=\id\}$.}

\noindent However the groups $\Aut_{\bf m}\,\na^3_K$ are not linear, even if $K$ is finite, as shown by

\begin{MainB} The group $\Aut_{\bf m}\,\na^3_K$ is not
linear, even over a ring.
\end{MainB}

We will now turn our attention to the
small subgroups, namely the finitely 
generated (f.g. in the sequel)
subgroups of $\Aut\,\na^2_K$. 
Let $\Gamma\subset \Aut\,\na^2_\Q$ be

\centerline{ $\Gamma=\langle S,T\rangle$, where
$S(x,y)=(y,2x)$ and $T(x,y)=(x,y+x^2)$.}

\noindent In \cite{Co},  Y. Cornulier asked 
about the existence
of nonlinear f.g. subgroups in
$\Aut\,\na^2_\C$. An answer is provided by 
the Assertion C.1 of the next

\begin{MainC} Let $K$ be a field of characteristic zero.

(C.1) The subgroup 
$\Gamma\subset\Aut_0\,\na^2_K$ is not linear, even over a ring.

(C.2) Any f.g. subgroup of
$\SAut_0\,\na^2_K$ is linear over $K(t)$.
\end{MainC}

It turns out that the group $\Gamma$, which is presented by 

\centerline{$\langle\sigma, \tau\mid \sigma^2\tau\sigma^{-2}=\tau^2\rangle$,}

\noindent appears in \cite{DS}
as the first example of a one-related group which is residually finite but not linear
over a field. 

By Theorem C.1, $\Gamma$ is also the first example of a  1-related group which
is not linear (even over a ring) but which
is embeddable in the automorphism group of an
algebraic variety. Residual finiteness 
of $\Gamma$ follows from general principles \cite{BL}, but the observation that $\Gamma$
acts on the finite sets $\F_p^2$, for any
odd $p$, and faithfully on their product, provides
a concrete proof.

For an infinite field $K$, Theorem A.2 
suggests to ask which groups
$G\supset \Aut_1\,\na^2_K$ are linear or not.
Indeed, either this group contains

\centerline{$\SAut\na^2_K:=\{\phi\in \Aut_0\na^2_K\mid \jac(\phi)=1\}$,}

\noindent
which is not linear, or it is ismorphic to
$\Aut_S\na^2_K$,
for some  subgroup $S$ of
$\GL(2,K)$. Therefore we ask

\centerline{\it For a given subgroup 
$S\subset \GL(2,K)$, is
the group $\Aut_S\na^2_K$ linear?}

\noindent For the subgroups $S$ of $\SL(2,K)$, 
three criteria provide an almost complete answer,
see Sections 8 and 9. Some examples of application  are

\begin{ExA} Let $q$ be a quadratic form on $K^2$ and $S=\SO(q)$. 

If $q$ is anisotropic, $\Aut_S\na^2_K$
is linear over a field extension of $K$.

Otherwise $\Aut_S\na^2_K$ is not linear, even over a ring.
\end{ExA}

\begin{ExB} For some cocompact lattices 
$S\subset \SL(2,\R)$,  
$\Aut_S\,\na^2_\C$ is linear over $\C$.

For any  lattice $S\subset \SL(2,\C)$, $\Aut_S\,\na^2_\C$ is not linear, even over a ring.
\end{ExB}

\begin{ExC} Let $d$ be a squarefree integer,
let ${\cal O}$ be the ring of integers of
$k:=\Q(\sqrt{d})$. 
Set $K=k(x)((t))$ and
$S=\SL(2,{\cal O}[x,x^{-1},
t])$.

If $d>0$, the group $\Aut_S\,\na^2_K$
is not linear, even over a ring.

Otherwise, $\Aut_S\,\na^2_K$
is linear over some field of characteristic zero. 
\end{ExC}

\bigskip\noindent
{\it 0.2 About the main points of the paper}

\noindent Since  {\it the topics are not ordered } as in the general introduction, a summary has been provided. Also note that {\it the statements in
the introduction  are often weaker} than those in the main text.

For a group $S$, we define
in Section 2 the notion of 
{\it a mixed product} of $S$ as a semi-direct product $S\ltimes *_{p\in P} E_p$, where 
$(E_p)_{p\in P}$ is a family of groups and
$S$ acts by permuting the factors of the free product. Hence $P$ is a $S$-set such that $E_p^s=E_{s.p}$ for any $s\in S$ and $p\in P$.

In section 3, it is shown that, under a mild  assumption, a mixed product (or an amalgamated product) which is linear over a ring is
automatically linear over a field. 
This  explains
the dichotomy {\it linear over a field/not linear, even over a ring} in our statements.
Then, we can use  the theory of algebraic groups 
to show that some mixed products are not linear, even over a ring. 

Let $S$ be a subgroup of $\SL(2,K)$.
The main question of the paper is to decide
if $\Aut_S\,\na^2_K$ is linear or not.
Indeed, this group
is a mixed product $S\ltimes *_{\delta\in\P^1_K}\,E_\delta(K)$.
Hence, there are two obstructions for the
linearity. 

First, the groups
$S_\delta\ltimes E_\delta(K)$ have to be linear 
with a uniform bound on the degree.
This problem is solved by using 
the notion of {\it semi-algebraic characters} for subgroups $\Lambda\subset K^*$. It
was inspired by the famous paper of  
Borel and Tits \cite{BT}, proving that the abstract  isomorphisms of simple
algebraic groups are semi-algebraic.
Strictly speaking, our paper only provides
a partial answer, because otherwise it would  had been too long. 

The second obstruction is the possibility, or not, to glue together some representations  
of the groups $S$ and
$S_\delta\ltimes E_\delta(K)$ to get
a representation of $\Aut_S\,\na^2_K$. 
Our linearity criterion is stronger in
characteristic zero than in finite characteristics. In characteristic zero, we can
use for the gluing process the following stronger version of 

\begin{MainC2} Assume $K$ of characteristic $0$.
There is an embbeding

\centerline{
$\SAut_0^{<n}\,\na^2_K\subset 
\SL(1+\lcm(1,2\dots,n),K(t))$.}
\end{MainC2}

\noindent Here $\SAut_0^{<n}\,\na^2_K$ denotes
the subgroup of $\SAut_0\,\na^2_K$ generated
by the automorphisms of degree $<n$,
for some $n\geq 3$.

The proof of the strong version of Theorem C.2 uses 
one trick, based on the Lie superalgebra
$\osp(1,2)$ and one idea, the ping-pong lemma. This idea   was originally invented 
by Fricke and Klein for the dynamic of groups with respect to the metric topologies, see  \cite{FN}. Later,  Tits used it in the context of the ultrametric topologies \cite{T72},
and we follow the Tits idea. 
Here, the ping-pong setting requires
a  representation of very large dimension.

As a brief conclusion  

$\bullet$ if $\ch K=0$, then 
$\Aut_0\,\na^2_K$ is not even locally 
linear,
$\SAut_0\,\na^2_K$ is locally 

linear but not linear, and $\Aut_1\,\na^2_K$ is linear. Therefore, the main questions 

arise for
the groups $G$ with $\Aut_1\,\na^2_K
\subset G\subset \SAut_0\,\na^2_K$.

$\bullet$ if $\ch K=p$,
$\SAut_0\,\na^2_K$ is linear iff $K$ is finite,
while $\Aut_1\,\na^2_K$ is always 

linear. In particular, 
$\Aut\,\na^2_{\overline \F_p}$
is locally linear but not linear. The exis-

tence of nonlinear f.g. subgroups is an open question.

$\bullet$ These phenomena are specific to dimension two.

\section{Main Definitions and Conventions}

Througout the whole paper, $K$ will denote a given field. Its {\it ground field} is $\F=\F_p$ if
$\ch\,K=p$ or $\F=\Q$ otherwise.

\bigskip
\noindent {\it 1.1 Group theoretical notations}

\noindent Let $S$ be a group and let  
$x,\,y\in S$. The symbols $y^x$ and $(x,y)$ are defined by

\centerline{$y^x:=xyx^{-1}$ and 
$(x,y):=xyx^{-1}y^{-1}$.} 

By definition, a {\it $S$-set} is a set $P$ endowed with an action of $S$. The {\it stabilizer} 
of a point $p\in P$ is the subgroup 
$S_p:=\{s\in S\mid s.p=p\}$.
The {\it core} $\Core_S(A)$ of a subgroup
$A$ of $S$  is the kernel of the action of $S$ on $S/A$.

Similarly, a
{\it $S$-group} is a group
$E$ endowed with a homomorphism $S\to \Aut(E)$. The corresponding semi-direct product of $S$ by $E$ is denoted by $S\ltimes E$.
Given another $S$-groups
$E'$, a homomophism, respectively an isomorphism,  
$\phi:E\to E'$ is called an {\it $S$-homomorphism},
resp. an {\it $S$-isomorphism} if it commutes with the $S$-action.

\bigskip
\noindent {\it 1.2 Commutative rings and group functors}

\noindent  Throughout the whole paper, a {\it commutative ring} means an associative commutative unital  ring.

A {\it group functor}  
is a functor $G:R \mapsto G(R)$ from the category of commutative   rings $R$ to the category of groups
( see  e.g. \cite{D} and  \cite{T89}
for the functorial approach to group theory).
The standard example of a group functor is
$R\mapsto \GL(n,R)$, where $n$ is a given positive integer.

Given an ideal $I$ of a commutative ring $R$, we denote by $G(I)$ the kernel of the homomorphism $G(R)\to G(R/I)$. It is called the {\it congruence subgroup} associated to the ideal $I$. For example, 
$\GL(n,I)$ is the  subgroup of $\GL(n,R)$ of all matrices of the form $\id+A$, where all entries of $A$ are in $I$.

In most cases, {\it we will only define the group $G(K)$ when $K$ is  a field} and the reader should understand that the definition over a ring is similar.

For the group functor  $K\mapsto\Aut\, \na_K^2$  and its consorts, our notation is not consistent, since 
$K$ is an index.

\bigskip
\noindent {\it 1.3 The group functors 
$\Elem_*(K)$,  $\Aff(2,K)$, and their subgroups}

\noindent 
By definition, an {\it elementary automorphism}  of 
 $\na_K^2$ is an automorphism 
 
 \centerline{$\phi:(x,y)\mapsto (z_1 x +t, z_2 y +f(x))$}
 
 \noindent for some $z_1, z_2\in K^*$, some $t\in K$ and some $f\in K[x]$. The group of elementary automorphisms of $\na_K^2$ is denoted $\Elem (K)$. Set

\hskip3cm $\Elem_0(K)=\Elem (K)\cap \Aut_0\,\na^2_K$, 
 
\hskip3cm $\SElem_0(K)=\Elem (K)\cap \SAut_0\,\na^2_K$, and
 
\hskip3cm $\Elem_1(K)=\Elem (K)\cap \Aut_1\,\na^2_K$.
 
\noindent However, {\it  we will   use the simplified notation $E(K)$ for $\Elem_1(K)$}.

Let $\Aff(2,K)$ 
be the subgroup
of affine automorphisms of $\na_K^2$.  Set 

\hskip2cm $B_{\Aff}(K):=\Aff(2,K)\cap \Elem(K)$,  

\hskip2cm $B_{\GL}(K):
=B_{\Aff}(K)\cap \GL(2,K)$, and 
 
\hskip2cm $B(K):
=B_{\Aff}(K)\cap \SL(2,K)$.

Indeed $B_{\Aff}(K)$, $B_{\GL}(K)$ and $B(K)$ are the standard  Borel subgroups of 
$\Aff(2,K)$, $\GL(2,K)$ and  $\SL(2,K)$. We have

\centerline{$\Aff(2,K)/B_{\Aff}(K)
=\GL(2,K)/B_{\GL}(K)=\SL(2,K)/B(K)
\simeq \P^1_K$, and}

\centerline{$\Elem_0(K)=B_{GL}(K)\ltimes E(K)$ and
$\SElem_0(K)=B(K)\ltimes E(K)$.}

\bigskip
\noindent {\it 1.4 An informal definition of finite-codimensional group subfunctors}

\noindent Informally speaking, a group subfunctor $H\subset G$ has {\it finite codimension} if the functor $R\mapsto G(R)/H(R)$ is "represented" 
by a  scheme $X$ of finite type. In particular, it means that there is a natural transform
$G/R\to X$ wich induces a bijection 
$G(K)/H(K)\to X(K)$ whenever $K$ is an algebraically closed field. Since this notion is used only for 
presentation purpose,  we will not  provide 
a formal definition. 

 For example, the subgroup $\Aut_0\,\na_K^2$   has codimension $2$ in $\Aut\,\na_K^2$, since the quotient $\Aut\,\na_K^2/\Aut\,\na_K^2$ is naturally $\na^2_K$. This fact does not require to 
involve the elusive theory of ind-algebraic groups.

 \section{Mixed Products}

Let $S$ be a group and let
$(E_p)_{p\in P}$ be a collection of groups
indexed by some $S$-set $P$.
 Since we did not found a name in the literature,
we will call {\it mixed product of $S$}
any semi-direct product $S\ltimes *_P E_p$
where $S$ acts on the free product $*_P E_p$ by permuting its factors, i.e. we have
$E_p^s=E_{s.p}$ for any $s\in S$ and $p\in P$ (see \cite{MKS} ch.4
for the definition of free products).

The connections between mixed products, amalgamated products and free products  are investigated.
As a consequence of  van der Kulk's Theorem,
we show that the groups  $\Aut_S\,\na^2_K$
are mixed products and that the groups
 $\Aut_1\,\na^2_K$ and $\Aut_{U(K)}\,\na^2_K$  are free products.

\bigskip
\noindent {\it 2.1  Amalgamated products}
 
\noindent  Let $A$, $G_1$ and $G_2$ be groups
and let $f_1:A\to G_1$ and $f_2: A\to G_2$ be
 group homomorphisms. Let $G_1 *_A G_2$ be 
the {\it amalgamated product} of $G_1$ and 
$G_2$ over $A$,
see e.g. \cite{S83}, ch. I. 
Since this product satisfies a universal property,  it is often called a
{\it free} amalgamated product, see \cite {MKS} ch.8.

In what follows, we will {\it always assume that 
$f_1$ and $f_2$ are injective}. Hence the homomorphisms 
 $G_1\rightarrow G_1 *_A G_2$ and 
 $G_2\rightarrow G_1 *_A G_2$ are
 injective, by  the Theorem 1 of  ch. 1 in \cite{S83}. Therefore, we will use a less formal terminology. The groups $G_1$ and $G_2$ are viewed as subgroups of  $G_1 *_A G_2$ and we will say that
 $G_1$ and $G_2$ {\it share $A$ as a common subgroup}.

  \bigskip \noindent 
{\it 2.2 Reduced words}

\noindent
The usual definition \cite{S83} of reduced words
is based on the right $A$-cosets. In order to avoid a confusion
between the set difference notation $A\setminus X$ 
and the $A$-orbits notation $A\backslash X$, we will
use a definition based on the left $A$-cosets.

Let $G_1$, $G_2$ be two groups sharing a common subgroup $A$,
and let $\Gamma=G_1 *_A G_2$.
Set $G_1^*=G_1\setminus A$,  $G_2^*=G_2\setminus A$
and let $T_1^*\subset G_1^* $ (respectively 
$T_2^*\subset G_2^* $) be a set of  representatives of  
$G_1^*/A$ (resp. of $G_2^*/A$).

Let $\Sigma$ be the set of all finite alternating sequences
${\boldsymbol\epsilon}=(\epsilon_1,\dots,\epsilon_n)$  of ones and twos. A {\it reduced word} of {\it type ${\boldsymbol\epsilon}$}  is a word $(x_1,\dots x_n,x_0)$ where $x_0$ is in $A$, and
$x_i\in T_{\epsilon_i}^*$ for any $1\leq i\leq n$.
Let ${\cal R}$ be the set of all reduced words.
The next Lemma is well-known, see  
 e.g. \cite{S83}, Theorem 1.

\begin{lemma}\label{words} The map

\centerline{ $(x_1,\dots x_n,x_0)\in {\cal R} 
 \mapsto  x_1\dots x_n x_0 \in G_1*_A G_2$}
 
 \noindent is bijective.
 \end{lemma}
 
 Set $\Gamma=G_1*_A G_2$. For $\gamma\in \Gamma\setminus A$
 there is some integer $n\geq 1$, some 
 ${\boldsymbol\epsilon}=(\epsilon_1,\dots,\epsilon_n)\in \Sigma$ and some $g_i\in G_{\epsilon_i}^*$
 such that $\gamma=g_1...g_n$. It follows from
 loc. cit.  that
 $\gamma=x_1...x_nx_0$ for some reduced word 
 $(x_1...x_nx_0)$ of type ${\boldsymbol\epsilon}$. Since it is determined by $\gamma$, the sequence ${\boldsymbol\epsilon}$ is called {\it the type} of $\gamma$.

 \bigskip \noindent 
{\it 2.3 Amalgamated product of subgroups}

\noindent
Let $G_1$, $G_2$ be two groups sharing a common subgroup $A$ and set $\Gamma=G_1*_A G_2$.
Let $G_1'\subset G_1$, $G_2'\subset G_2$
and $A'\subset A$ be subgroups such that 

\centerline{$G_1'\cap A=G'_2\cap A=A'$.}

 \begin{lemma}\label{subamal} (i) The natural map 
 $G_1'*_{A'} G_2'\to \Gamma$ is injective. 
 
 (ii) Let $\Gamma'\subset \Gamma$ be a subgroup
 such that $\Gamma'.A=\Gamma$. Then  we have
 
\centerline{$\Gamma'=G_1'*_{A'} G_2'$,}

 \noindent where $G_1'= G_1\cap \Gamma'$, 
 $G_2'= G_2\cap \Gamma'$ and $A'= A\cap \Gamma'$. 
 \end{lemma}

\begin{proof} {\it Proof of Assertion (i).}
For $i=1,2$ set $G_i^*=G_i\setminus A$,
$G_i^{'*}=G'_i\setminus A'$. Let 
$T^*_i\subset G_i^*$ and $T^{'*}_i\subset G_i^{'*}$
be a set of representatives of
$G_i^*/A$ and $G_i^{'*}/A'$.

Since the maps
$G_i'/A'\to G_i/A$ are injective, it can be assumed that  $T^{'*}_i\subset T^{*}_i$. Let $\cal R$ 
and $\cal R'$ be the set of reduced words of
$G_1*_A G_2$, and respectively of $G_1'*_{A'} G_2'$.
By definition, we have $\cal R'\subset \cal R$, thus
by Lemma \ref{words} the map $G_1'*_{A'} G_2'\mapsto G_1*_A G_2$
is injective. 

\smallskip\noindent
{\it Proof of Assertion (ii).}
We will use the notations of the previous proof.
Since $\Gamma'.A=\Gamma$, it follows that 
the  maps $G'_1/A'\to G/A$ and $G_2'/A'\to G_2/A$ are bijective. Therefore ${\cal R'}$ is the set of
all reduced words $(x_1,\dots,x_n,x_0)\in{\cal R}$ such that $x_0\in A'$. It follows easily that

\centerline{
$G_1*_A G_2/G_1'*_{A'} G_2'\simeq A/A'=\Gamma/\Gamma'$,}

\noindent and therefore we have 
$\Gamma'=G_1'*_{A'} G_2'$.
\end{proof}

\bigskip
\noindent
{\it 2.4 The group $\Aut\,\na^2_K$ is an amalgamated product}
 
\noindent Indeed, it is the classical  

\begin{vdKthm}\cite{vdK} We have 

\centerline{$\Aut\,\na_K^2\simeq \Aff(2,K)*_{B_{\Aff}(K)}\Elem(K)$.}
\end{vdKthm}
 
\bigskip 
\noindent {\it 2.5 Mixted products}

\noindent 
Let $S$ be a group, let $P$ be a $S$-set and
let $Q\subset P$ be a set of representatives of $P/S$. A mixed product $S\ltimes *_
{p\in P}\, E_p$ satisfies the following universal property.
 
 \begin{lemma}\label{universal} Let
 $\Gamma\supset S$ be a group. Assume given, for any $q\in Q$, a $S_q$-homomorphism
 $\phi_q:E_q\to\Gamma$.
 Then there is a unique group homomorphism 
 
 \centerline{$\phi:S\ltimes *_{p\in P}\, E_p\to 
\Gamma$}
 
 \noindent such that
 $\phi\vert_{S}=\id$ and $\phi\vert_{E_q}=\phi_q$
 for any $q\in Q$.
  \end{lemma}

  \begin{proof}
  
  Let us define, for any $p\in P$, a $S_p$-homomorphism
$\phi_p:E_p\to \Gamma$  as follows.
Let $s\in S$ such that $q:=s.p$ belongs to $Q$. Set

\centerline{$\phi_p(u)=s^{-1}\phi_q(sus^{-1}) s$,}

\noindent
for any $u\in E_p$. Since $\phi_q$ is a
$S_q$-homomorphism, the defined homomorphism
$\phi_p$  only depends on $s$ modulo $S_p$. 
Moreover the collection  of homomorphisms
$(\phi_p)_{p\in P}$ induces a $S$-homomorphism
from $*_{p\in P}\, E_p$ to $\Gamma$,
which extends to the required homomorphism
$\phi:S\ltimes *_{p\in P}\, E_p\to 
 \Gamma$.
\end{proof}
 
 It follows that a mixed product 
 $S\ltimes *_{p\in P}\, E_p$ is entirely determined by $S$ and the the  $S_q$-groups $E_q$ for $q\in Q$.
 For the record, let us state

\begin{lemma}\label{mixing3} 
Let $\Gamma=S\ltimes *_{p\in P}\, E_p$ and
$\Gamma'=S\ltimes *_{p\in P}\, E'_p$ be two mixed groups.

If for any $q\in Q$, the groups $E_q$ and $E'_q$ are $S_q$-isomorphic, then the groups $\Gamma$ and
$\Gamma'$ are isomorphic.
\end{lemma}

\bigskip
\noindent
{\it 2.6 Mixed products with a transitive action on $P$}

\noindent In this subsection, we show that
the mixed products with a transitive action of $S$ on $P$ are the amalgamated products $S*_A\,G$ where $A$ is a retract in $G$.

First, let $S$, $G$ be two groups sharing a common subgroup $A$ with the additional assumption that $A$ is a retract in $G$.
Therefore, we have $G=A\ltimes E$,
for some normal subgroup $E$ of $G$. Set
$\Gamma=S*_{A} G$ and let $\Gamma_1$ be the kernel
of the map $\Gamma\to S$ induced by the retraction
$G\to A\simeq G/E$. It is clear that

\centerline{$\Gamma=S\ltimes \Gamma_1$.}

\begin{lemma}\label{mixing1} 
Let $P$ be a set of representatives of $S/A$.
We have  

\centerline{$\Gamma_1\simeq *_{\gamma\in P} E^\gamma$.} 
 
In particular 
$S*_A G$ is isomorphic to the mixed product
$S\ltimes*_{\gamma\in P}\, E^\gamma$.
\end{lemma}

\begin{proof} We can assume that $1\in P$.
By Lemma \ref{universal}, there is  a unique homomorphism 
$\phi:S\ltimes*_{\gamma\in P}\, E^\gamma\to\Gamma$
such that its restriction to $E^1=E$ and to $S$ is the identity. Conversely, the group $S\ltimes*_{\gamma\in P}\, E^\gamma$ contains the subgroups $S$ and $G\simeq A\ltimes E^1$ whose intersection is $A$. Hence, the universal property of amalgamated products provides a natural homomorphism 
$\psi:\Gamma\to\ltimes*_{\gamma\in P}\, E^\gamma$.
Clearly, $\phi$ and $\psi$ are inverses of each other,
what shows the lemma.
\end{proof}

Conversely, let  $\Gamma=S\ltimes *_{p\in P} E_p$ be a mixed product. 

\begin{lemma}\label{mixing2} Assume that $S$ acts transitively on $P$. Then we have

\centerline{$S\ltimes (*_{p\in P} E_p)\simeq 
S*_{S_q} (S_q\ltimes E_q)$,}

\noindent where $q$ is any chosen point in $P$.
\end{lemma}

The proof of the  Lemma \ref{mixing2} will be skipped.
Indeed it is based on universal properties, as the previous proof.

\bigskip
\noindent
{\it 2.7 The group $\Aut_S\,\na^2_K$ is a mixed product}

\noindent For a subgroup  $S$ of $\GL(2,K)$,
recall that 

\centerline{$\Aut_S\,\na^2_K:=
\{\phi\in\Aut_0\,\na_K^2\mid 
\d\phi_{\bf 0}\in S\}$.}

\noindent As usual, a line $\delta\in  \P^1_K$
has {\it projective coordinates} $(a;b)$ if 
$\delta=K.(a,b)$. For such a $\delta$, 
let $E_\delta(K)\subset \Aut\,\na^2_K$
be the subgroup

\centerline{$E_{\delta}(K):=\{(x,y)\mapsto 
(x,y)+f(bx-ay) (a,b)\mid f\in t^2K[t]\}$.}

\noindent Let $\gamma\in \GL(2,K)$ such that
$\gamma.\delta_0=\delta$  where 
$\delta_0\in \P^1_K$ has coordinates $(0;1)$.
Then we have 
$E_{\delta_0}(K)=E(K)$ and 
$E_{\delta}(K)=E(K)^\gamma$.

\begin{lemma}\label{mixed} We have

\centerline{$\Aut_S\,\na^2_K\simeq S\ltimes\,
*_{\delta\in  \P^1_K} E_\delta(K)$.}
\end{lemma}

\begin{proof} Clearly, it is enough to prove the statement for $S=\GL(2,K)$.
Since $B_{\Aff}(K)$ contains the translations, we have
$B_{\Aff}(K).\Aut_0\,\na^2_K=\Aut\,\na^2_K$. Therefore by
van der Kulk's Theorem and
Lemma \ref{subamal}, we have

\centerline
{$\Aut_0\,\na^2_K\simeq \GL(2,K)*_{B_{\GL}(K)}\Elem_0(K)$.}

Since $\Elem_0(K)=B_{\GL}(K)\ltimes E(K)$, it follows from Lemma \ref{mixing1} that

\hskip2.6cm
$\Aut_0\,\na^2_K\simeq \GL(2,K)\ltimes\,
*_{\delta\in  \P^1_K} E_\delta(K)$.
\end{proof}

\bigskip
\noindent
{\it 2.8 Mixed product with an almost free transitive action on $P$}

\noindent The action of $S$ on a set $P$ is called
{\it almost free transitive} if $P$ consists of a fixed point and a free orbit under $S$. (It will be tacitly assumed that $S\neq 1$, so the fixed point and the free orbit are well defined.) In this subsection we show
that the mixed products with an almost free transitive action of $S$ on $P$ are the free products
$G*G'$, where $S$ is a retract in $G$.

First let $\Gamma=S\ltimes *_{p\in P} E_p$ be a mixed product. 

\begin{lemma}\label{almostft1} Assume that the action of $S$ on $P$ is
almost free transitive. Then $\Gamma$ is isomorphic to the free product

\centerline{
$(S\ltimes E_{p_0})*E_{p_\infty}$,}

\noindent where $p_0\in P$ is the fixed point and
$p_\infty\in P$ is any point of the free orbit.
\end{lemma}

Conversely let $\Gamma=(S\ltimes E)*F$ be a free product, where $E$ is a $S$-group and  $F$ is another group. 

\begin{lemma}\label{almostft2} The group $\Gamma$ is isomorphic to the the mixed product

\centerline{$S\ltimes (E* (*_{s\in S} F_s))$,}

\noindent where $F_s$ denotes a copy of $F$, 
for any $s\in S$.
\end{lemma}

The easy proofs of the previous two lemmas, 
which follow the same pattern as Lemma
\ref{mixing3}, will be skipped.

\bigskip
\noindent
{\it 2.9 The group $\Aut_{U(K)}\,\na^2_K$ is a free product}

\noindent Recall that  $U(K)$ is the group of linear transforms $(x,y)\mapsto (x,y+ax)$, for some
$a\in K$.
Let $\delta_0,\,\delta_\infty\in\P^1_K$ be the points with projective coordinates $(0;1)$ and $(1;0)$. The group $E_{\delta_0}(K)=E(K)$ commutes with $U(K)$.

\begin{lemma}\label{free2} We have

\centerline{
$\Aut_{U(K)}\,\na^2_K\simeq (U(K)\times E_{\delta_0}(K))*E_{\delta_\infty}(K)$.}
\end{lemma}

\begin{proof} By Lemma \ref{mixed}, we have
$\Aut_{U(K)}\,\na^2_K\simeq U(K)\ltimes\,
*_{\delta\in  \P^1_K} E_\delta(K)$. Since the action
of $U(K)$ on $\P^1_K$ is almost free transitive, the assertion  follows from Lemma \ref{almostft1}.
\end{proof}

\bigskip\noindent
{\it 2.10 A corollary}

\begin{Cor}\label{Card}
Let $K$, $L$ be  fields such that
$\Card\, K=\Card\, L$ and $\ch\, K=\ch\, L$.
We have

\centerline {$\Aut_1\,\na^2_K\simeq \Aut_1\,\na^2_L$
and $\Aut_{U(K)}\,\na^2_K\simeq 
\Aut_{U(L)}\,\na^2_L$.}
\end{Cor}

\begin{proof} It can be assumed that $K$ is infinite.
Let $\F$ be its prime subfield and let
$E$ be a $\F$-vector space with
$\dim_\F\,E=\aleph_0\,[K:\F]=\Card\,K$.

By Lemma \ref{mixed}, 
$\Aut_1\,\na_K^2$ is a free product of 
$\Card K$ copies of $E$, from which its follows
that $\Aut_1\,\na_K^2$ only depends on
the cardinality and the characteristic of the field $K$, hence we have $\Aut_1\,\na^2_K\simeq \Aut_1\,\na^2_L$.

The proof that $\Aut_{U(K)}\,\na^2_K\simeq 
\Aut_{U(L)}\,\na^2_L$ is identical.
\end{proof}

\section{Linearity over Rings vs. over Fields}

\noindent In this section, we show Corollaries \ref{criterion1} and 
\ref{criterion2}. They state 
that, under a mild assumption, a mixed product, 
or an amalgamated product,
which is linear over a ring, is 
automatically linear over a field.

\bigskip 
\noindent {\it 3.1 Linearity Properties}

\noindent
For a group,  the linearity over a field
is the strongest linearity property. Besides  the case of 
{\it prime} rings, i.e. the subrings of a field,
a group which is linear over a ring $R$ is not necessarily linear over a field. Two 
relevant examples are provided in the subsection 8.6.

On the opposite, there are    groups 
containing a f.g.  subgroup which is
not linear, even over a ring. These groups are nonlinear in the strongest sense.

\bigskip\noindent
{\it 3.2 Minimal embeddings}

\noindent Let $R$ be a commutative ring and
let $\Gamma$ be a subgroup of $\GL(n,R)$
for some $n\geq 1$. 
The embedding $\Gamma\subset \GL(n,R)$ is called {\it minimal} if
for any  ideal $J\neq \{0\}$ we have $\Gamma\cap \GL(n,J)\neq \{1\}$.

\begin{lemma} \label{MinEmb} Let $\Gamma$ be a subgroup of  $\GL(n,R)$. For some ideal $J$, the induced homomorphism $\Gamma\to\GL(n,R/J)$ is a minimal embedding.
\end{lemma}

\begin{proof} Since $R$ could be non-noetherian,
the  proof requires Zorn's Lemma.

Let $\cal S$ be the set of all ideals $J$
of $R$ such that $\Gamma\cap \GL(n,J)=\{1\}$. With respect to the inclusion, ${\cal S}\ni \{ 0 \}$ is a nonempty poset.  For any chain $\cal C\subset\cal S$, the ideal 
$\cup_{I\in{\cal C}}\,I$ belongs to $\cal S$. Therefore Zorn's Lemma implies that $\cal S$ contains a maximal element $J$. It follows that the induced homomorphism $\Gamma\to\GL(n,R/J)$ is a minimal embedding.
\end{proof}

\bigskip\noindent
{\it 3.3 Groups with  trivial normal centralizers}

\noindent 
By definition, a group $\Gamma$ 
 has {\it the trivial normal centralizers property} if, for any subset $S\not\subset\{1\}$, its centralizer
 $C_\Gamma(S)$ is not normal, except if $C_\Gamma(S)$
 is the trivial group. Equivalently, if $H_1$ and $H_2$  are commuting normal subgroups of $\Gamma$, then one of them is trivial.

\begin{lemma}\label{linearity} Let  $\Gamma$ be a  group with the trivial normal centralizers property.

If $\Gamma$ is linear over a ring, 
then $\Gamma$ is also linear over a field.
\end{lemma} 

\begin{proof} By hypothesis and
Lemma \ref{MinEmb}, there exists a minimal
embedding $\rho:\Gamma\subset \GL(n,R)$ for some commutative ring $R$.
The case $\Gamma=\{1\}$ can be  excluded, so we will assume that $R\neq \{0\}$.

Let $I_1,\,I_2$ be ideals 
of $R$ with $I_1.I_2=0$. Since $H_1:=\Gamma\cap \GL(n,I_1)$ and  
$H_2:=\Gamma\cap \GL(n,I_2)$
are commuting  normal subgroups of $\Gamma$,  one of them is trivial. Since $\rho$ is minimal, $I_1$ or $I_2$ is the zero ideal. Thus $R$ is prime.

It follows that $\Gamma\subset \GL(n,K)$, where 
$K$ is the fraction field of $R$.
\end{proof}

\bigskip\noindent
{\it 3.4 Amalgamated products with a trivial core}

\noindent
Let  $G_1$, $G_2$ be two groups sharing a common subgroup $A$ and set $\Gamma=G_1*_A G_2$.

 Let $\Sigma$ be the set of all finite alternating sequences
 ${\boldsymbol \epsilon}=(\epsilon_1,\dots,\epsilon_n)$ of ones and twos.
For $i,\,j\in\{1,2\}$, let $\Sigma_{i,j}$ be
be the subset of all ${\boldsymbol \epsilon}=(\epsilon_1,\dots,\epsilon_n)\in \Sigma$
starting with $i$ and ending with $j$ and 
let $\Gamma_{i,j}$ be the set of all $\gamma\in \Gamma$
of type ${\boldsymbol \epsilon}$ for some ${\boldsymbol  \epsilon}\in \Sigma_{i,j}$.
Therefore we have
 
\centerline{ $\Gamma=A\sqcup \Gamma_{1,1}
 \sqcup \Gamma_{2,2}\sqcup \Gamma_{1,2}\sqcup \Gamma_{2,1}$.}

 By definition, the amalgamated product 
 $G_1*_A\,G_2$  is called {\it nondegenerate} if
 $G_1\neq A$ and $G_2\neq A$. It is called 
{\it dihedral} if $G_1=G_2=\Z/2\Z$, and $A=\{1\}$,
and {\it nondihedral} otherwise.

 \begin{lemma}\label{conj1} Let $\Gamma=G_1*_A\,G_2$ be a 
 nondegenerate and nondihedral
 amalgamated product such that $\Core_\Gamma(A)$ is trivial.
 
 For any element $g\neq 1$ of $\Gamma$,
 there are 
 $\gamma_1,\,\gamma_2\in \Gamma$ such that
 
 \centerline {$g^{\gamma_1}\in \Gamma_{1,1}$ and
 $g^{\gamma_2}\in \Gamma_{2,2}$.}
 
In particular $\Gamma$ has the trivial normal centralizers property.
 \end{lemma}

 \begin{proof}

  First it should be noted that $A$ cannot be 
 simultaneously a subgroup of index 2 in $G_1$ and in $G_2$.
 Otherwise the core hypothesis  implies that 
 $A=\{1\}$
 and $\Gamma$ would be the dihedral group.
  Hence we can assume that $G_2/A$ contains
 at least $3$ elements. 
 
 Next it is clear that
 $G^*_i.\Gamma_{j,k}\subset \Gamma_{i,k}$
 and $\Gamma_{k,j}.G^*_i\subset \Gamma_{k,i}$
  whenever $i\neq j$. 
 
\smallskip\noindent 
 {\it Proof  that  the conjugacy class of any $g\neq 1$ intersects both $\Gamma_{1,1}$ and $\Gamma_{2,2}$.} Let  $\gamma_1\in G_1^*$ and $\gamma_2\in G_2^*$.
 We have 
 $\Gamma_{2,2}^{\gamma_1}\subset\Gamma_{1,1}$ and
 $\Gamma_{1,1}^{\gamma_2}\subset\Gamma_{2,2}$. Therefore the claim is proved for
 any $g\in \Gamma_{1,1}\cup \Gamma_{2,2}$. Moreover
 it is now enough to prove that the conjugacy class of any
 $g\neq 1$ intersects $\Gamma_{2,2}$. 
 
 Assume now $g\in \Gamma_{2,1}$. We have $g=u.v$ for some
 $u\in G^*_2$ and $v\in\Gamma_{1,1}$. Since 
 $[G_2:A]\geq 3$, there is $\gamma\in G^*_2$ such that
 $\gamma.u\notin A$. It follows that
 $\gamma.g$ belongs to $\Gamma_{2,1}$, and therefore
 $g^\gamma$ belongs to $\Gamma_{2,2}$. 
 
 For $g\in \Gamma_{1,2}$, the claim follows from the fact that $g^{-1}$ belongs to $\Gamma_{2,1}$.
 
 Last, let $g\in A\setminus\{1\}$. Since
 $\Core_A(\Gamma)$ is trivial,  there is $\gamma\in \Gamma$ such that $g^\gamma$ is not in $A$. Thus $g^\gamma$ belongs to 
 $\Gamma_{i,j}$  for some
 $i,\,j$. So $g$ is conjugate to some element in
 $\Gamma_{2,2}$ by the previous considerations.
 
\smallskip\noindent 
 {\it Proof  
that $\Gamma$ has the trivial normal centralizers property.}
 Let $H_1,\,H_2$ be nontrivial normal subgroups. By the previous point, there are elements $g_1,\,g_2$ with
 
 \centerline{$g_1\in H_1\cap \Gamma_{1,1}$ and
 $g_2\in H_2\cap \Gamma_{2,2}$.}
 
\noindent Since we have $g_1 g_2\in \Gamma_{1,2}$ and $g_2 g_1\in \Gamma_{2,1}$, it follows that $g_1g_2\neq g_2g_1$. Therefore
$H_1$ and $H_2$ do not commute.
\end{proof}

\begin{Cor}\label{criterion1} Let $\Gamma=G_1*_A\,G_2$ be a nondegenerate amalgamated product such that
$\Core_\Gamma(A)$ is trivial
\footnote{As it has been noticed by the referee, this is equivalent to the faithfulness of the action of $\Gamma$ on the associated  Bass-Serre tree.}.

If $\Gamma$ is linear over a ring, then $\Gamma$ is linear over a field.
\end{Cor} 

\begin{proof} Since the infinite dihedral group is linear over a field, we will assume that the amalgamated product 
$\Gamma=G_1*_A\,G_2$ is also nondihedral. Thus
the result is an obvious corollary of Lemmas \ref{linearity} and \ref{conj1}.
 \end{proof}

 \bigskip\noindent
{\it 3.5 Mixed products with trivial core}

\noindent Let $S$ be a group, and 
let $S\ltimes *_{p\in P} E_p$ be a mixed product of $S$.

Let $\Sigma$ be the set of all finite  sequences
 ${\boldsymbol\pi}=(p_1,\dots,p_m)$ of elements of $P$
 with $p_i\neq p_{i+1}$, for any $i<m$.
 Set $\Gamma_1=*_{p\in P} E_p$ and
 $E^*_p=E_p\setminus\{1\}$.
 Any element $u\in \Gamma_1\setminus\{1\}$ is uniquely written
 as $u=u_1\dots u_m$, where $u_i\in E^*_{p_i}$
for some $m\geq 1$ and some sequence 
${\boldsymbol\pi}=(p_1,\dots,p_m)\in\Sigma$. The decomposition $u=u_1\dots u_m$ is called the
{\it reduced decomposition} of $u$, 
${\boldsymbol \pi}$ is called its {\it type}
and $m$ is called its {\it length}. 
 For $p,p'\in P$, let 
$E_{p,p'}$
be the set of all elements 
$u\in \Gamma_1\setminus\{1\}$ 
whose type is a sequence ${\boldsymbol\pi}$ starting with 
$p$ and ending with $p'$.

 By definition, the free product 
 $ *_{p\in P} E_p$, 
or, by extension, the mixed product $S\ltimes *_{p\in P} E_p$,  is called {\it nondegenerate} if $\Card\,P\geq 2$ and 
 $E_p\neq \{1\}$ for any $p\in P$. 
For a nondegenerate mixed product 
$S\ltimes *_{p\in P} E_p$,  
we have  

\centerline{$\Core_\Gamma(S)=
\Core_\Gamma(\cap_P\, S_p)$.}

 The mixed product 
$S\ltimes *_{p\in P} E_p$ is called 
{\it dihedral} if $\Card\,P=2$, 
if $E_p=\Z/2\Z$ for any $p\in P$ and if

\centerline{$S\simeq \Z/2Z$ permutes the two factors,
or $S=\{1\}$.}

\noindent It is called {\it nondihedral} otherwise.

 \begin{lemma}\label{conj2}
 Let $\Gamma=S\ltimes *_{p\in P} E_p$ be a nondegenerate and nondihedral mixed product such that 
 $\Core_\Gamma(\cap_P\,S_p)=\{1\}$. Let $p\in P$.
 
\noindent (i)  For any element $\gamma\in \Gamma_1\setminus\{1\}$,
 there is $v\in \Gamma_1$ such that $\gamma^{v}$
 belongs to $E_{p,p}$.

 \noindent (ii) For any element $\gamma\in \Gamma\setminus\Gamma_1$, there is $v\in \Gamma_1$ such that $(\gamma,v)\neq 1$.

In particular $\Gamma$ has the trivial normal centralizers property.
 \end{lemma}
 
 \begin{proof} 
 
 \smallskip\noindent 
 {\it Proof of Assertion(i).}
For $\Card P=2$, the group $\Gamma_1$ is not the infinite dihedral group $\Z/2\Z*\Z/2\Z$ and 
 the assertion follows from Lemma \ref{conj1}. Therefore, we will assume that $\Card P\geq 3$.
 
 The element $\gamma$ belongs to $E_{p_1,p_2}$ for some $p_1,\,p_2\in P$. Let 
 $p_3\in P\setminus\{p_1,\,p_2\}$
 and let $v\in E_{p,p_3}$.
Thus the element  $\gamma^v$ belongs to 
$E_{p,p}$.

  \smallskip\noindent 
 {\it Proof of Assertion (ii).}
 Let $\gamma=su$, where $s\in S\setminus\{1\}$ and $u\in \Gamma_1$. 
 
 Obviously $\Gamma/S$ and $\Gamma_1$ are isomorphic $S$-sets. Since
 $\Core_\Gamma(S)$ is trivial,
 there is $t\in \Gamma_1$ such that 
 $(s,t)\neq 1$.
Thus, we can assume that $u\neq 1$. 

Let $u=u_1\dots u_m$ be its reduced decomposition, let $(p_1,\dots,p_m)$ be its type. Let $v\in E_{p'}^*$ with $p'\neq p_m$.
 By definition, $u_1\dots u_m.v.u_m^{-1}\dots u_1^{-1}$
is a reduced decomposition of the element 
$w:=uvu^{-1}$. Hence $w$ and  $w^s$ have lenght $2m+1>1$. Thus 
we have $v^\gamma=w^s\neq v$, or equivalently 
$(\gamma,v)\neq 1$.

\smallskip\noindent 
{\it Proof that $\Gamma$ has the trivial normal centralizers property.} Let $H,\,H'$ be two nontrivial normal subgroups
of $\Gamma$. Let $p\neq p'$ be elements of $P$.
By  Assertions (i) and (ii), there are elements $g, \,g'$ with
 
 \centerline{$g\in H\cap E_{p,p}$ and
$g'\in H'\cap E_{p',p'}$.}
 
\noindent Since we have $gg'\in E_{p,p'}$ and 
$g'g\in E_{p',p}$, it follows that $gg'\neq g'g$. Therefore $H$ and $H'$ do not commute.
\end{proof}

\begin{Cor}\label{criterion2} Let $\Gamma=S\ltimes *_{p\in P} E_p$ be a nondegenerate mixed product such that 
 $\Core_\Gamma(\cap_P\,S_p)=\{1\}$.

If $\Gamma$ is linear over a ring, then $\Gamma$ is linear over a field.
\end{Cor} 

\begin{proof} Since the infinite dihedral group is linear over a field, we can assume that the mixed product 
$\Gamma=S\ltimes *_{p\in P} E_p$ is also nondihedral. 
Then the result is an obvious corollary of Lemmas \ref{linearity} and \ref{conj2}.
 \end{proof}

\section{A Nonlinear f.g. Subgroup of $\Aut_0\,\na_\Q^2$}

Let $\Gamma$ be the group with presentation

\centerline{
$\left\langle \sigma,\tau \mid  
\sigma^2 \tau \sigma^{-2}=\tau^2
\right\rangle$.}

\noindent
In \cite{DS}, C. Drutu and M. Sapir showed that
$\Gamma$ is not linear over a 
field\footnote{ In the first version of this paper that appeared in the arXiv, I was  unaware of \cite{DS}. I'm grateful to T. Delzant for providing this reference.}.
We show that $\Gamma$ is not linear either over a ring, and
that $\Gamma$ is isomorphic to 
an explicit subgroup of $\Aut_0\,\na_\Q^2$, which proves Theorem A.2.

\bigskip
\noindent {\it 4.1 The amalgamated decomposition 
$\Gamma=G_1*_A G_2$}

\noindent Let us consider the following subgroups of $\Gamma$

\centerline{$G_1=\langle\sigma\rangle$,
$G_2=\langle\sigma^2,\tau\rangle$ and 
$A=\langle\sigma^2\rangle$.}

\noindent The groups $G_2$
is isomorphic to $\Z\ltimes\Z[1/2]$ where
any $n\in\Z$ acts over $\Z[1/2]$ by multiplication by $2^n$. The group $\Gamma$  is the amalgamated product

\centerline{ $\Gamma\simeq G_1*_A G_2$.}

\begin{lemma}\label{prepa1} The group
$\Gamma$ has the trivial normal centralizers property.
\end{lemma} 

\begin{proof} Set $H=\Z[1/2].\tau$.
The $A$-sets $G_2/A$ and $H$ are isomorphic,
hence $A$ acts faithfully on $G_2/A$. Therefore 
$\Core_{\Gamma}(A)\subset \Core_{G_2}(A)$ is trivial, and the assertion follows from Lemma \ref{conj1}.
\end{proof}

\bigskip
\noindent {\it 4.2 Quasi-unipotent endomorphisms}

\noindent Let $V$ be a finite-dimensional vector space over an algebraically closed field $K$. An element $u\in \GL(V)$ is called
{\it quasi-unipotent} if all its eigenvalues are roots of unity.
The {\it quasi-order} of a quasi-unipotent endomorphism
$u$ is the smallest positiver integer $m$ such that 
$u^m$ is unipotent.

If $u$ is unipotent and $\ch\,K=0$, set

\centerline
{$\log\,u=\log(1-(1-u)):=\sum_{k\geq 1} \,(1-u)^k/k$,}

\noindent which is well-defined since $1-u$ is nilpotent.

\begin{lemma}\label{unipotent} Let $h, u\in \GL(V)$.
Assume that $u$ has infinite order and
$h u h^{-1}=u^{2}$.
Then  $u$ is quasi-unipotent of  quasi-order $m$
for some odd integer $m$.
Moreover $K$  has characteristic zero, and

\centerline{$h e h^{-1}=2 e$,}

\noindent where
$e:=\log\,u^m$.
\end{lemma}

\begin{proof} Let $\Spec\,u$ be the spectrum of $u$.
By hypothesis the  map $\lambda\mapsto\lambda^2$
is bijective on $\Spec\,u$.
Hence  all 
eigenvalues of $u$ are odd roots of unity, what proves that $u$ is
quasi-unipotent of odd quasi-order $m$.

Over any field of finite characteristic, the unipotent
endomorphisms have finite order. Hence we have $\ch\,K=0$. Moreover, we have

\hskip1cm
$h e h^{-1}=h (\log u^m) h^{-1}=
\log(h u^m h^{-1})=\log(u^{2m})=2e$.
\end{proof}

\noindent{\it 4.3 Nonlinearity of $\Gamma$}

\begin{DSLemma}\label{DS} The group $\Gamma$ is not linear over a field.
\end{DSLemma}

The result is a particular case of Corollary 4 in
\cite{DS}. Since their proof  is based on an earlier result of \cite{We}, we shall provide a direct proof.

\begin{proof} Assume otherwise and
let $\rho':\Gamma\to \GL(V)$ be an embedding, where $V$ is a finite-dimensional vector space over an algebraically closed field $K$.
Since $\tau$ has infinite order and $\sigma^2 \tau \sigma^{-2}=\tau^2$, it follows from 
Lemma \ref{unipotent} that $K$ has characteristic zero,  $\rho'(\tau)$ is quasi-unipotent of
odd quasi-order $m$.  

\smallskip\noindent
{\it Step 1: there is another embedding 
$\rho:\Gamma\to \GL(V)$ such that 
$\rho(\tau)$ is unipotent.}
Let   
$\psi: \Gamma\to\Gamma$ be the 
group homomorphism 
defined by 
$\psi(\sigma)=\sigma$,
and $\psi(\tau)=\tau^m$.
Since $\psi(G_1)\cap \psi(G_2)=A$, it follows from
Lemma \ref{subamal} that the natural
homomorphism $\psi(G_1)*_A \psi(G_2)\to\Gamma$ is injective. Hence
$\psi$ is injective and 
$\rho:=\rho'\circ\psi$ is an embedding
such that $\rho(\tau)=\rho'(\tau)^m$ is unipotent.

\smallskip\noindent
{\it Step 2: the unipotent subgroup 
$U\subset \GL(V)$.}
Set $h=\rho(\sigma^2)$,  let $\Pi=\Spec\,h$ be its spectrum, and for each $\lambda\in \Pi$, let 
$V_{(\lambda)}$ be the corresponding generalized eigenspace. For any $k\geq 0$, set

\centerline{
$\Pi_\geq k=
\{\lambda \in \Pi\mid
\lambda\in 2^l \,\Pi$ for some
$l\geq k\}$,}

\noindent
The filtration $\Pi=\Pi_{\geq 0}\supset
\Pi_{\geq 1}\supset \dots$ of the set $\Pi$ induces
a filtration of $V$

\centerline{$V=V_{\geq 0}
\supset V_{\geq 1}\supset\dots$,}

\noindent where
$V_{\geq k}=\oplus_{\lambda\in\Pi_{\geq k}}V_{(\lambda)}$.  Let $U$ be the
group of all $g\in \GL(V)$ such that
$(g-\id) V_{\geq k}\subset V_{\geq k+1}$ for all $k\geq 0$. For some suitable basis,  $U$ is a group of upper triangular matrices. Therefore $U$ is nilpotent.

\smallskip\noindent
{\it Step 3: $\rho(\Gamma)$ is nilpotent by commutative.} Since $\rho(G_1)$ commutes with $h$,
we have $\rho(G_1).V_{\geq k}=V_{\geq k}$ for any integer $k$.
Therefore $\rho(G_1)$ normalizes $U$.

Set $u=\rho(\tau)$ and 
$e=\log\,u$. By Lemma \ref{unipotent}, we have 
$h e h^{-1}=2e$ and therefore we have 
$e.V_{\geq k}\subset V_{\geq k+1}$. It follows that 
$u=\exp\,e$ belongs to $U$. Since 
$\rho(\Gamma)=\langle \rho(G_1), u \rangle$
 we have

\centerline{
$\rho(\Gamma)\subset \rho(G_1)\ltimes U$,}

\noindent and therefore $\rho(\Gamma)$ is nilpotent by commutative. Hence $\rho(\Gamma)$  contains a nontrivial
normal abelian subgroup.
This contradicts  Lemma \ref{prepa1}, which states that 
$\Gamma$ has the trivial normal
centralizers property. 
 \end{proof}

\begin{lemma} \label{nonlinear}

The group $\Gamma$ is not linear, even over a ring.
\end{lemma}

\begin{proof} 
By Lemma \ref{prepa1} the group $\Gamma$
has the trivial normal centralizers property. It follows
from  Lemmas \ref{linearity} and \ref{DS} that
$\Gamma$ is not linear, even over a ring.
\end{proof}

\noindent{\it 4.4 Proof of Theorem C.1}

\begin{MainC1} The subgroup $\langle S,T\rangle$
of $\Aut_0\,\na_\Q^2$ is not linear, even
over a ring, where 

\centerline{$S(x, y) = (y, 2x)$ and $T (x, y) = (x, y + x^2 )$.}
\end{MainC1}

\begin{proof} Set $H_1=\langle S\rangle$, $H_2=\langle S^2,T\rangle$, $C=\langle S^2\rangle$.

We have $S^2=2.\id$, therefore  we have
 $H_1\cap B_{\Aff}(K)=C$. Moreover $H_2$ is the group of automorphisms of the form
 
\centerline{ $(x,y)\mapsto (2^k x, 2^k y+rx^2)$,}

\noindent for $k\in\Z$ and  $r\in \Z[1/2]$, therefore
 $H_2\cap B_{\Aff}(K)=C$. It follows from Lemma \ref{subamal}
 and  van der Kulk's Theorem that the natural homomorphism
 $H_1*_C H_2\to\Aut_0\,\na_\Q^2$ is injective. 
 
 There is a
 group isomorphism $\Gamma\to H_1*_C H_2$ sending $\sigma$ to $S^{-1}$ and $\tau$ to $T$. 
 Thus, by Lemma \ref{nonlinear}, the subgroup
 of $\Aut_0\,\na_K^2$ generated by $S$ and $T$ is not linear, even over a ring.
\end{proof}

 \section{The Linear Representation of $\Aut_{1}\,\na_K^2$}
 
\noindent We will prove Theorem A.2, in a way 
which is useful for Section 9.

\bigskip\noindent
{\it 5.1 Nagao's Theorem} 

\noindent For a subgroup
$S$ of $\GL(2,K)$, set

\centerline{$\GL_S(2,K[t])=\{G(t)\in \GL(2,K[t]) \mid
G(0)\in S\}$.}

\begin{Nagaothm}  We have

\centerline{$\GL(2,K[t])\simeq
\GL(2,K)*_{B_{\GL}(k)}
\GL_{B_{\GL}(k)}(2,K[t])$.}
\end{Nagaothm}

\bigskip
\noindent
{\it 5.2 The group $\GL_S(2,K[t])$ is a mixed product}

\noindent  For any $\delta\in\P^1_K$, let $e_\delta\in\End(K^2)$ be a nilpotent element
with $\Image\, e_\delta=\delta$. For any commutative
$K$-algebra $R$, set

\centerline{$U_\delta(R):=\{ \id + r e_\delta \mid r\in R\}$.}

\noindent Obviously, $U_\delta(R)$ is a subgroup of
$\SL(2,R)$. Let $\gamma\in \GL(2,K)$ such that
$\gamma.\delta_0=\delta$  where 
$\delta_0\in \P^1_K$ has coordinates $(0;1)$.
We have 

\centerline{$U_{\delta_0}(R)=U(R)$ and 
$U_{\delta}(R)=U(R)^\gamma$.}

\begin{lemma}\label{mixed2} Let $S$ be a subgroup of $\GL(2,K)$. We have

\centerline
{$\GL_S(2,K[t])\simeq 
S\ltimes\,*_{\delta\in\P^1_K}
U_\delta(tK[t])$.}
\end{lemma}

\begin{proof} Clearly, it is enough to prove the lemma
for $S=\GL(2,K)$. Since 

\centerline{
$\GL_{B_{\GL(k)}(k)}(2,K[t])=
B_{\GL(k)}(k)\ltimes U(tK[t])$}

\noindent the lemma follows from Nagao's Theorem and Lemma \ref{mixing1}.
\end{proof}

\bigskip
\noindent
{\it 5.3 The groups $\SL(2,tK[t])$
and $\SL_{U(K)}(2,K[t])$ are free products}

\noindent For $S\subset\SL(2,K)$, the group
$\GL_{S}(2,K[t])$ lies in $\SL_{S}(2,K[t])$. Thus set

\centerline{$\SL_{S}(2,K[t]):=\GL_{S}(2,K[t])$.}

\noindent
Let $\delta_0,\,\delta_\infty$ be the points in $\P^1_K$ with  coordinates 
$(0;1)$ and $(1;0)$.

\begin{lemma}\label{Tits}  We have 

\centerline{$\SL(2,tK[t])= *_{\delta\in\P^1_K}\,U_\delta(tK[t]$,
and}

\centerline{$\SL_{U(K)}(2,K[t])=U([K[t]) *U_{\delta_\infty}(tK[t])$.} 
\end{lemma}

\begin{proof} The first assertion follows from Lemma
\ref{mixed2}. Since the action of $U(K)$ on
$\P^1_K$ is almost free transitive, the second point follows from Lemma \ref{almostft1}.
\end{proof}

\smallskip
\noindent {\it Remark.}  The group
$\SL_{U(K)}(2,K[t])$ is the "lower nilradical" of the
affine Kac-Moody group $\SL(2,K[t,t^{-1}])$.
In \cite{T82}, Tits defined the "lower nilradical" of any Kac-Moody group in term of an inductive limit,
which  is essentialy equivalent to the previous lemma for $\SL_{U(K)}(2,K[t])$.
 
Since the notes \cite{T82} are not widely distributed, let us mention that an equivalent result is stated in \cite{T89}, Section 3.2 and 3.2, see also \cite{T87}.

\bigskip\noindent
{\it 5.4 Proof of Theorem A.2}

\begin{lemma}\label{iso} There are isomorphisms

\centerline{
$\Aut_1\,\na_K^2\simeq \SL(2,tK[t])$
and $\Aut_{U(K)}\,\na_K^2\simeq \SL_{U(K)}(2,K[t])$.}
\end{lemma}

\begin{proof}  

By Lemmas \ref{mixed} and \ref{Tits}, 
$\Aut_1\,\na_K^2$ and
 $\SL(2,tK[t])$ are free products of 
 $\Card \P^1_K$ copies of  a $K$-vector space
 of dimension $\aleph_0$. Therefore these two groups are isomorphic.
 
 The proof of the isomorphism 
 $\Aut_{U(K)}\,\na_K^2\simeq \SL_{U(K)}(2,K[t])$
 follows similarly from
 Lemmas \ref{free2} and \ref{Tits}.
\end{proof}

\begin{MainA2}
The groups $\Aut_1\,\na_K^2$ and $\Aut_{U(K)}\,\na_K^2$ 
embed in $\SL(2,K(t))$.

Moreover if $K\supset k(t)$ for some infinite field 
$k$, then there exists an embedding 
 $\Aut_1\,\na_K^2\subset \Aut_{U(K)}\,\na_K^2\subset \SL(2,K)$.
\end{MainA2}

\begin{proof} It follows from Lemma \ref{iso} that
$\Aut_1\,\na_K^2$ and $\Aut_{U(K)}\,\na_K^2$ are
subgroups of $\SL(2,K(t)$, and therefore 
they are linear over $K(t)$.

Assume now that  $K\supset k(t)$ for some infinite field $k$. We claim that there exists a field $L$ with
$L(t)\subset K$ and $\Card\,L=\Card\,K$. If
$\Card\, K=\aleph_0$, then the subfield $k$ satisfies the claim. Otherwise, we have $\trdeg\,K> \aleph_0$
and there is an embedding $L(t)\subset K$ for some
subfield $L$ with $\trdeg\,L=\trdeg\,K$. Since
$\trdeg\,L=\Card L=\Card K$, the claim is proved.

It follows from Corollary \ref{Card} that 

\centerline{$\Aut_{U(K)}\,\na_K^2\simeq \Aut_{U(K)}\,\na_L^2\subset \SL(2,L(t))
\subset\SL(2,K)$,} 

\noindent therefore 
$\Aut_1\,\na_K^2$ and $\Aut_{U(K)}\,\na_K^2$ are
subgroups of $\SL(2,K)$. 
\end{proof}

\bigskip\noindent
{\it 5.5 A Corollary}

\noindent For a a finite field $K$, 
$\Aut_1\,\na_K^2$ has finite index in
$\Aut\,\na_K^2$, hence

\begin{Cor}  For a a finite field $K$, the
group $\Aut\,\na_K^2$ is linear over $K(t)$.
\end{Cor}

\section{The Linear Representation of $\SAut_{0}^{<n}\,\na_K^2$}

For $n\geq 3$, let
$\SAut_{0}^{<n}\,\na_K^2$
 be the subgroup of
$\SAut_0\,\na^2_K$
generated by all automorphisms
$\phi\in \SAut_0\,\na^2_K$ of degree $<n$.
In this section, {\it we will  assume that $K$ has characteristic zero}, in order to
show that the group
$\SAut_{0}^{<n}\,\na_K^2$ is linear, what
proves Theorem C.2.  Unfortunately, our approach does not extend to fields of finite characteristic.

For a nonzero vector-valued polynomial $v(t)=\sum\,v_i t^i$, let $\deg v$ be its degree and let 
$\hdc(v):=v_{\deg v}$ be its {\it highest degree component}.

\bigskip
\noindent
{\it 6.1 A ping-pong lemma}

\noindent Let $S$ be a group and 
let $\Gamma=S\ltimes*_{p\in P}\,F_p$ be a mixed product of $S$.

\begin{lemma}\label{ping-pong} Let $\Omega$ be a $\Gamma$-set,
and let  $(\Omega_p)_{p\in P}$ be a collection of subsets 
in $\Omega$.
Set $F^*_p=F_p\setminus\{1\}$ and assume

(i) the free product $*_{p\in P}\,F_p$ is nondegenerate and nondihedral,

(ii) $\Core_\Gamma(\cap_P\,S_p)$ is trivial,

(iii) the subsets $\Omega_p$ are nonempty and disjoint, and

(iv) we have $F^*_p.\Omega_q\subset \Omega_p$ whenever $p\neq q$.

Then $\Gamma$ acts faithfully on $\Omega$.
\end{lemma}

\begin{proof} 
Let $p\neq p'$
be two elements in $P$. By Lemma \ref{conj2}, Assertion (i) any nontrivial
normal subgroup of $\Gamma$ 
contains  some $\gamma\in E_{p,p}$. 
We have $\gamma.\Omega_{p'}\subset\Omega_p$,
therefore $\gamma$ acts nontrivially on $\Omega$.
Since any nontrivial normal subgroup  acts nontrivially,
the action of $\Gamma$ on $\Omega$ is faithful.
 \end{proof}

\bigskip\noindent
{\it 6.2 The group  $\SAut_{0}^{<n}\,\na_K^2$ is a mixed product with trivial core}

\noindent
Let $\delta\in\P^1$ with  coordinates 
$(a;b)$. For $n\geq 3$, 
let $E_\delta^{<n}(K)\subset E_\delta(K)$  
be the subgroup  of all automorphisms of the form
$(x,y)\mapsto (x,y)+f(bx-ay)(a,b)$ where  
$f(t)\in t^2K[t]$ and $\deg f(t)<n$.

\begin{lemma}\label{mixed3} For any $n\geq 3$,
the group
$\SAut_{0}^{<n}\,\na_K^2$ is isomorphic to the nondegenerate mixed product

\centerline{$\Gamma:=\SL(2,K)\ltimes*_{\delta\in\P^1_K}E_\delta^{<n}(K)$.}

\noindent Moreover $\Core_\Gamma(\SL(2,K))$ is trivial.
\end{lemma}

\begin{proof} Let  
$u\in *_{\delta\in\P^1_K}E_\delta(K)$ with reduced
decomposition $u_1\dots u_m$,
where $u_i\in E_{\delta_i}(K)$.
By induction over $n$, it is easy to prove
simultaneously  that $\deg\,u=\prod\,\deg\,u_i$
and that $\hdc(u)$ is of the form
$(x,y)\mapsto (bx-ay)^{\deg u} (c,d)$,
where $(c;d)$ and $(a;b)$
are some  coordinates of
$\delta_1$ and $\delta_n$.

By Lemma \ref{mixed}, 
$\SAut_{0}\,\na_K^2$ is isomorphic to the mixed product

\centerline{$\SL(2,K)\ltimes*_{\delta\in\P^1_K}E_\delta(K)$.}

\noindent Any  $\phi\in \SAut_{0}\,\na_K^2$
decomposes uniquely as $\phi=s u_1\dots u_m$, where
$s\in\SL(2,K)$ and $u_1\dots u_m$ is a reduced decomposition in $*_{\delta\in\P^1_K}E_\delta(K)$.
Since $\deg \phi=\prod\,\deg\,u_i$, we have 
$\SAut_{0}^{<n}\,\na^2_K=
\langle \SL(2,K), E_\delta^{<n}(K)\rangle$. Thus

\centerline{$\SAut_{0}^{<n}\,\na_K^2\simeq\SL(2,K)\ltimes*_{\delta\in\P^1_K}E_\delta^{<n}(K)$.}

\noindent It has been noticed that 
$\Core_\Gamma(\SL(2,K))$ acts trivially on 
$\P^1_K$, hence it is included in  
$\{1,\sigma\}$, where $\sigma(x,y)=(-x,-y)$. Set $\tau(x,y):=(x,y+x^2)$.
Since $\tau^{\sigma}(x,y)=(x,y-x^2)$, it follows that  $\Core_\Gamma(\SL(2,K))$ is trivial.
\end{proof}

\bigskip
\noindent
{\it 6.3 The square root $\eta$ of $e$}

\noindent Let $\epsilon$ be an odd variable. For an integer $N\geq 1$, let
$L(N)\subset K[x,y]$ and
${\hat L}(N)\subset K[x,y,\epsilon]$
be the subspaces of homogenous polynomials of degree $N$. Let $(e,h,f)$ be the usual basis
of $\fsl(2,K)$. As an $SL(2,K)$-module, we have ${\hat L}(N)=L(N)\oplus L(N-1)$
and $e$ acts as the derivation 
$x{{\partial}\over{\partial y}}$. Set

\centerline{
$\eta=x{\partial\over\partial \epsilon} +\epsilon {\partial\over\partial y}$.}

\noindent It is clear that $\eta^2=e$.
Indeed $\hat L(N)$ is a representation of the Lie superalgebra $\osp(1,2)$, and 
$\eta\in \osp(1,2)$ is an odd element such that
$\eta^2=e$.

For any
$\delta\in\P^1_K$ with projective coordinates 
$(a;b)$, set $L_\delta:=K.(ax+by)^N$ and
$L_\delta^*=L_\delta\setminus\{0\}$. Let 
$\delta_0, \delta_\infty\in\P^1_K$ be the points with 
projective coordinates $(0;1)$ and $(1;0)$.
Since 
$\eta^{2N}.y^N=
(x{{\partial}\over{\partial y}})^N
y^N=N!\,x^N$, it follows that

\begin{lemma}\label{Jordan2}
We have $\eta^{2N}.L_{\delta_{\infty}}^*
\subset L_{\delta_0}^*$.
\end{lemma}

\bigskip
\noindent
{\it 6.4 The representation $\rho_N$ of
$\SAut_0\,\na^2_K$ on 
${\hat L}(N)\otimes K[t]$}

\noindent We will extend the natural representation of
of $\SL(2,K)$ on ${\hat L}(N)\otimes K[t]$
to 
$\SAut_0\,\na^2_K$ as follows. For any
 automorphism $\tau\in E(K)$, set 

\centerline{$\rho_N(\tau)=\exp (t\eta f(\eta))$,}

\noindent if 
$\tau(x,y)=(x,y+f(x))$, where 
$f(x)\in x^2K[x]$.  
Since $[e,\eta]=0$ and $[h,\eta]=\eta$, the homomorphism
$\rho_N$ is $B(K)$-equivariant. By 
Lemma \ref{universal}, $\rho_N$ extends to a
$K[t]$-linear action of $\SAut_0\,\na^2_K$.

\begin{lemma}\label{faith}
Assume that $2N$ is divisible by
$\lcm(1,2,\dots, n)$. 

Then the restriction
of $\rho_N$ to $\SAut_{0}^{<n}\,\na_K^2$ is faithful.
\end{lemma}

\begin{proof} For any $\delta\in \P^1_K$, set
$F_\delta^*=E_{\delta}^{<n }(K)\setminus\{1\}$ and

\centerline{$\Omega_\delta=\{v(t)\in {\hat L}(N)\otimes K[t]\setminus\{0\}\mid \hdc(v(t))\in L_\delta^*\}$.} 

\noindent {\it First step.} Let 
$\delta_0, \delta_\infty\in\P^1_K$ be as in Lemma
\ref{Jordan2}.
We claim that 

\centerline{$F_{\delta_{0}}^*.\Omega_{\delta_\infty}\subset 
\Omega_{\delta_0}$.}

\noindent Let
$\tau(x,y)=(x,y+f(x))$ be in 
$F_{\delta_{0}}^*$.
We have
$f(x)=ax^k+$ higher terms, for some $a\in K^*$
and some $k$ with  
$2\leq k<n$. By definition, we have

\centerline{
$\rho_N(\tau)=\exp t\eta f(\eta)=
\sum_{m\geq 0}\, {\eta^m f(\eta)^m\over m!}\, t^m$.}

\noindent Since $\eta f(\eta)$ is divisible 
by $\eta^{k+1}$ and $\eta^{2N+1}=0$, it follows that $\eta^m f(\eta)^m=0$ for $m>2N/(k+1)$. Since $k+1$ divides $2N$, $\rho_N(\tau)$ is a polynomial of degree $d:=2N/(k+1)$ and we have

\centerline{$\hdc(\rho_N(\tau))=
{a^{d}\over d!}\,\eta^{2N}$. }

\noindent Let $v(t)\in\Omega_{\delta_{\infty}}$. 
By Lemma
\ref{Jordan2}, $\hdc(\rho_N(\tau)).\hdc(v(t))$ is nonzero and belongs to $L_{\delta_0}^*$. It follows that 
$\rho_N(\tau).
\Omega_{\delta_\infty}\subset\Omega_{\delta_0}$, what proves the claim.

\smallskip
\noindent{\it Second step: use of the ping-pong lemma.}
Let $\delta\neq \delta'$ in $\P^1_K$. Since
$F_\delta^*\times\Omega_{\delta'}$ is conjugate
under  $\SL(2,K)$ to
$F_{\delta_0}^*\times\Omega_{\delta_\infty}$, the previous result implies that

\centerline{$F_\delta^*.\Omega_{\delta'}\subset 
\Omega_{\delta}$.}

\noindent
By Lemmas \ref{ping-pong} and \ref{mixed3}, the restriction
of $\rho_N$ to $\SAut_{0}^{<n}\,\na_K^2$ is faithful.
\end{proof}

\bigskip
\noindent{\it 6.5 Proof of Theorem C.2}

\noindent Since $\dim {\hat L}(N)=2N+1$,   Lemma \ref{faith} implies that

\begin{MainC2} 

For any $n\geq 3$, there is an embedding

\centerline{$\SAut_{0}^{<n}\,\na_K^2\subset 
\SL(1+\lcm(1,2,\dots,n), K(t))$.}

In particular, any f.g. subgroup of
$\SAut_{0}\,\na_K^2$ is linear over $K(t)$.

\end{MainC2}

\section{Semi-algebraic Characters}

 Let
$\Lambda\subset K^*$ be a subgroup. For any
$n\geq 1$, let $K_n\subset K$ be the subfield
generated by $\Lambda^n$. 
Let $L$ be  an algebraically closed field,
which contains at least one subfield isomorphic to $K_1$ and let $\F$ be the ground field of $K$.

For $n\geq 1$,
a group homomorphism $\chi: \Lambda\to L^*$
is called a {\it semi-algebraic character}
of degree $n$ if $\chi(z)=\mu(z^n)$
for some field embedding $\mu:K_n\to L$.
Let ${\cal X}_n(\Lambda)$ be the set of
all semi-algebraic characters of $\Lambda$
of degree $n$. The degree of a semi-algebraic character is {\it not uniquely defined}.
Given $n\neq m$, we will show a criterion for the disjointness of ${\cal X}_n(\Lambda)$ and
${\cal X}_m(\Lambda)$.

\bigskip
\noindent
{\it 7.1 The invariant $I_n(\Lambda)$}

\noindent  Let $\F[\Lambda]$
be the group algebra of $\Lambda$.
Given a field $E\supset\F$, any  homomorphism 
$\chi:\Lambda\to E^*$ extends to an
algebra homomorphism 
${\hat\chi}:\F[\Lambda]\to E$. Set

\centerline{$\Ker\,{\hat\chi}:=
\{\sum_{\lambda}\,a_\lambda \lambda\in
\F[\Lambda]
\mid \sum_\lambda\,a_\lambda \chi(\lambda)=0\}$.}

\noindent For $n\geq 1$, let $\chi_n$ be the  homomorphism
$\chi_n:\lambda\in\Lambda \mapsto\lambda^n\in K_n^*$. Set

\centerline{$I_n(\Lambda)=\Ker\,{\hat\chi}_n$.}

\begin{lemma} \label{invariant}  A group homomorphism
$\chi: \Lambda\to L^*$
is  a  semi-algebraic character
of degree $n$ iff 
$\Ker\,{\hat\chi}=I_n(\Lambda)$.

In particular, we have
${\cal X}_n(\Lambda)={\cal X}_m(\Lambda)$ or
${\cal X}_n(\Lambda)
\cap{\cal X}_m(\Lambda)=\emptyset$, for any
 positive integers  $n\neq m$.
\end{lemma}

\begin{proof}
By definition, the fraction field of the prime ring 
$\F[\Lambda]/I_n(\Lambda)$ is $K_n$. Hence
${\hat\chi}$ factors through $K_n$, i.e.
${\hat\chi}=\mu\circ{\hat\chi}_n$ for
some field embedding $\mu:K_n\to L$.
The first point follows, as well as the second.
\end{proof}

\bigskip
\noindent
{\it 7.2 Minimally bad subgroups of $K^*$}

\noindent Let $\Lambda\subset K^*$ be a subgroup.
By definition, the {\it transcendental
degree} of $\Lambda$ is
 $\trdeg\,\Lambda:=\trdeg K_1$ and its
  {\it rank} is
$\rk\,\Lambda:=\dim\,\Lambda\otimes\Q$.
Both are cardinals and we have
$\trdeg\,\Lambda\leq \rk\,\Lambda$. 
We say that $\Lambda$ is a {\it good subgroup}
of $K^*$ if $\trdeg\,\Lambda'= \rk\,\Lambda'$, for any
f.g. subgroup $\Lambda'$ of $\Lambda$ and a {\it bad subgroup} otherwise.

Assume now that $\Lambda$ is a  free abelian group of rank $r<\infty$,
with basis $x_1,\dots,x_r$. The ring 
$\F[\Lambda]$  is isomorphic to the ring
$\F[x_1^{\pm1},\dots x_r^{\pm1}]$ of Laurent polynomials. 
For $\alpha=(\alpha_1,\dots,\alpha_r)\in\Z^r$,
set  $x^\alpha=x_1^{\alpha_1},\dots x_r^{\alpha_r}$. 
The {\it support} of a Laurent polynomial
$P=\sum_{\alpha\in\Z^r}\,a_\alpha x^\alpha$
 is the set

\centerline{$\Supp\,P:=\{\alpha\in\Z^r\mid a_\alpha\neq 0\}$.}

\noindent Assume now that $\trdeg\,\Lambda=r-1$.
Since $\F[x_1^{\pm1},\dots x_r^{\pm1}]$  is
a unique factorization domain, $I_1(\Lambda)$ 
is  a principal ideal. If $P$ be one of its generator, 
the other generators are the polynomials $ax^{\gamma}P$, for $a\in \F^*$ and $\gamma\in\Z^r$.
Hence  the subgroup  
$X(\Lambda)\subset\Z^r$ generated by
 $\alpha-\beta$
for $\alpha,\,\beta\in \Supp\,P$
only depends on $\Lambda$. Moreover,
if $0\in \Supp\,P$,  then we have
$X(\Lambda)=\langle \Supp\,P\rangle$.

A subgroup $\Lambda\subset K^*$ is called 
{\it minimally bad} if

(i) $\Lambda$ is a f.g. free abelian group,  and 

(ii) we have $\rk\,\Lambda=1+\trdeg\,\Lambda$ and $X(\Lambda)=\Z^r$, where $r=\rk\,\Lambda$.

\begin{lemma}\label{minimal}
Let $\Lambda\subset K^*$ be  a
 bad subgroup of $K^*$.

Then $\Lambda$ contains a minimally bad
subgroup $\Lambda'$.
\end{lemma}

\begin{proof} By definition, $\Lambda$ contains a f.g. bad subgroup $\Lambda_0$. Moreover, we can assume that $\Lambda_0$ is torsion free.

Let ${\cal C}$ be the set
of all  subgroups
$\Pi\subset \Lambda_0$  such that
 $\rk\,\Pi>\trdeg\,\Pi$. 
 Let us pick one element $\Pi'$ of ${\cal C}$
 of minimal rank $r$. It is clear
 that $\rk\,\Pi'=1+\trdeg\,\Pi'$.
 Let $x_1,\dots x_r$ be a basis of $\Pi'$ and
 let $P=\sum_{x\in \Z^r} a_\alpha x^{\alpha}$
 be a generator of the ideal $I_1(\Pi')$ of 
$\F[\Pi']\simeq \F[x_1^{\pm1},\dots x_r^{\pm1}]$
such that $a_0\neq 0$. 

It is clear that
$\Lambda':=\langle x^{\alpha}\mid \alpha\in \Supp\,P\rangle$
 is a minimally bad subgroup.
 \end{proof}

\bigskip
\noindent
{\it 7.3 The Newton polygone of $P_n$}

\noindent Let $r\geq 1$ be an integer.  
Let $P$ be a generator of a principal ideal 
$I$ of
$\F[x_1^{\pm1},\dots x_r^{\pm1}]$. 
The {\it Newton polygone} $\New(P)$ of $P$ is the convex closure of $\Supp\,P$ in
$\R\otimes\Z^r$, and let $\Ext(P)$ be its 
set of
 extremal points. 
Up to translation by $\Z^r$, $\Ext(P)$
is an invariant of $I$. Hence
the largest integer $e(P)$ such that
$\alpha-\beta\in e(P).\Z^r$  for any
$\alpha,\,\beta\in \Ext(P)$ only depends on $I$.

It will convenient to choose an ordering of 
$\Z^r$. A Laurent polynomial 
$P\in \F[x_1^{\pm1},\dots x_r^{\pm1}]$ is  {\it normalized} if $a_0=1$ and any 
$\alpha\in \Supp\,P$ is nonnegative. 
Any principal ideal $I$ of $\F[x_1^{\pm1},\dots x_r^{\pm1}]$ has a unique normalized generator
$P$. Since $0$ belongs to $\Ext(P)$, we have 
$\Ext(P)\subset e(P).\Z^r$.

Let $\Lambda\subset K^*$ be a minimally bad
subgroup and let $x_1,\dots,x_r$ be a basis
of $\Lambda$. For any $n\geq 1$, let $P_n$
be the normalized generator of $I_n(\Lambda)$.

\begin{lemma}\label{newton} Assume that $n\geq 1$ is not divisible by $\ch\,K$.
Then we have

\centerline{$\New(P_n)={n^{r-1}\over f_n}              \New(P_1)$,}

\noindent for some integer $f_n$ dividing 
$e(P_1)$.
\end{lemma}

\begin{proof} In 
$\overline{\F}[x_1^{\pm1},\dots x_r^{\pm1}]$
the polynomial $P_1$ decomposes uniquely as

\centerline{$P_1=Q_1\dots Q_k$}

\noindent where  $Q_1,\,Q_2,\dots$ are normalized irreducible polynomials in 
${\overline\F}[x_1^{\pm1},\dots x_r^{\pm1}]$. Since they are permuted by $\Gal(\F)$, we have
$\Supp Q_1=\Supp Q_2...$, hence

(i) $\Supp Q_1$ generates $\Z^r$, and

(ii) $\New(Q_1)={1\over k}\New (P_1)$ and  
$k$ divides $e(P_1)$.

Let $\mu_n\subset {\overline \F}^*$ be the group of all $n$th root of one. 
For various
$(\zeta_1,\dots \zeta_r)\in\mu_n^r$,
the normalized polynomials
$Q_1(\zeta_1 x_1,\dots,\zeta_r x_r)$ are
pairwise distinct by the assertion (i). Thus
 the polynomial
$R=\prod_{\zeta_1,\dots,\zeta_r\in
\mu_n^r}\,Q_1(\zeta_1 x_1,\dots,
\zeta_r x_r)$ is irreducible in
$\overline{\F}[x_1^{\pm n},\dots x_r^{\pm n}]$.

Set $G_1=\{\sigma\in\Gal(\F)
\mid Q_1^\sigma=Q_1\}$ and
$G=\{\sigma\in\Gal(\F)
\mid R^\sigma=R\}$.
Since $R$ is normalized,
the polynomials $R^\sigma$ for 
$\sigma\in\Gal(\F)/G$ are pairwise distinct, and
therefore 
$S=\prod_{\sigma\in \Gal(\F)/G}\,R^\sigma$ is a
normalized  irreducible 
polynomial in 
$\F[x_1^{\pm n},\dots x_r^{\pm n}]$. 
It follows that 

\centerline{$P_n(x_1^n,\dots,x_r^n)=
S(x_1,\dots,x_r)$.}

\noindent Since $[\Gal(\F):G]=k$, the integer  
$f_n:=[G:G_1]$ divides $e(P_1)$, and

\centerline{
$\New(P_n)=1/n\,\New(S)={n^{r-1}\over f_n}\New(P_1)$.}
\end{proof}

\bigskip
\noindent
{\it 7.4 A criterion for the disjointness of
${\cal X}_n(\Lambda)$ and ${\cal X}_m(\Lambda)$}

\begin{lemma}\label{disjointness} Let $\Lambda\subset K^*$ be  a bad subgroup of $K^*$.
Then there is an integer $e\geq 1$ such that 

\centerline{${\cal X}_n(\Lambda)\cap 
{\cal X}_m(\Lambda)=\emptyset$,}

\noindent whenever the integers $n\neq m$ are coprime to $e$.
\end{lemma}

\begin{proof} By Lemma \ref{minimal},
$\Lambda$ contains a minimally bad subgroup $\Lambda'$.
Moreover by Lemma \ref{invariant}, it is enough
to show the lemma for $\Lambda'$. Therefore
we can assume that $\Lambda$ itself is minimally bad.
For any $n\geq 1$, let $P_n$ be the normalized generator of $I_n(\Lambda)$, relative to some order of $\Lambda$.

\smallskip\noindent{\it Proof when $\rk\,\Lambda=1$.} Thus $\F=\Q$ and a  generator $\lambda$ of $\Lambda$ is
an algebraic number of infinite order. Let $n,\,m\geq 1$ be two integers
such that $P_n=P_m$. Since $P_n(\lambda^m)=0$,
there is $\sigma\in\Gal(\Q)$ such that
$\lambda^m=\sigma(\lambda^n)$. Let $k\geq 1$ be
an integer such that $\sigma^k(\lambda)=\lambda$.
Therefore, we have

\centerline{
 $\lambda^{n^k}=
 \sigma^k(\lambda^{n^k})=\lambda^{m^k}$.}
 
\noindent Since $\lambda$ has infinite order, it follows that $n=m$. In this case, the lemma is
proved for $e=1$.

\smallskip\noindent{\it Proof when 
$r:=\rk\,\Lambda\geq 2$.} Set $e=pe(P_1)$ if 
$\ch\,K=p$ and $e=e(P_1)$ ortherwise.
Let $n, \,m\geq 1$ be two  integers coprime
to $e$ such that $P_n=P_m$. By Lemma \ref{newton},
we have ${n^{r-1}\over f_n}={m^{r-1}\over f_m}$
for some integers $f_n$ and $f_m$ dividing
$e$. It follows that $n=m$.
\end{proof}

\section{A Nonlinearity Criterion for $\Aut_S\,\na_K^2$}

Let $S_0$ be a subgroup of $B(K)$. We will use
Lemma \ref{disjointness} to give 
a necessary condition for the linearity
over a field of $S_0\ltimes E(K)$. Then we derive 
a nonlinearity criterion for 
the groups $\Aut_S\,\na_K^2$.

From now on, let  $\rho:S_0\ltimes E(K)\to 
\GL(V)$ be a a given embedding, 
where $V$ is a finite-dimensional vector space over an algebraically closed field $L$. Let $W(\rho)\subset \End\,V$ be the linear subspace generated by $\rho(E(K))$.

The commutative group structure on  $E(K)$
will be denoted additively.
Indeed $E(K)$ has a natural structure of
a graded vector space over $K$, namely
$E(K)=\oplus_{n\geq 3}\,K.T_{n}$ where 
$T_n(x,y)=(x,y+x^{n-1})$.

For any $g\in B(K)$, set  $\chi_{B}(g)=\lambda$
if  $g(x,y)=(\lambda^{-1}x,\lambda y+tx)$, for some 
 $t\in K$. The group of 
all eigenvalues of elements in $S_0$ is 
 $\Lambda:=\chi_B(S_0)\subset K^*$.
 We have $gT_ng^{-1}=\chi_B(g)^n T_n$.
 Since the action of $S_0$ on $E(K)$
 factors through $\Lambda$, it follows that
 $E(K)$ and  $W(\rho)$ are $\Lambda$-modules.

\bigskip
\noindent{\it 8.1 An obvious estimate}

\noindent A integer $n\geq 1$ is called a
{\it divisor of $\Lambda$} if $\Lambda$ contains a primitive $n$th root of one. Let $d(\Lambda)$
be the number, finite or infinite, of divisors of $\Lambda$.

\begin{lemma}\label{estimate} We have 
$d(\Lambda)\leq 2+ (\dim V)^2$.
\end{lemma}

\begin{proof}
For each  divisor $n$ with $n\geq 3$, set
$t_n=\rho (T_n)$. Since $g t_n g^{-1}=t_n$
iff $\chi(g)^n=1$, it follows that the elements
$t_n$ are linearly independant. Thus we have
$\dim\End\,V\geq d(\Lambda)-2$, from which the assertion follows.
\end{proof}

\smallskip
\noindent{\it 8.2 Unipotent representations}

\noindent 

\begin{lemma}\label{unirep}
Assume that $\rk \Lambda\geq 1$.

Then $\ch L=\ch K$ and $\rho(E(K))$ is a unipotent group.
\end{lemma}

\begin{proof} If $\ch K=p$, 
the assertions follow from the fact that $E(K)$ is an elementary $p$-group of infinite rank. 

We will now assume that $\ch K=0$.
Let $V=\oplus_{\chi\in\Omega}\,V_{(\chi)}$ be the 
generalized weight decomposition of
the $E(K)$-module $V$, where $\Omega$ is the set of
group homomorphisms $\chi:E(K)\to L^*$ such that
$V_{(\chi)}\neq 0$. 

Let $\chi\in\Omega$. Since $\Omega$ is finite,
the group $S_0'=\{s\in S_0\mid \chi^s=\chi\}$ has
finite index in $S_0$. There is some $s\in S_0'$
such that $\chi_B(s)$ has infinite order.
It follows that the map 
$e\in E(K)\mapsto e-ses^{-1}\in E(K)$ is invertible.
Therefore $\chi$ is
trivial and $\rho(E(K))$ is a unipotent group. 

Since $E(K)$ is torsion-free, $L$ has characteristic $0$.
\end{proof}

\smallskip
\noindent{\it 8.3 A linearity criterion for
$S_0\ltimes E(K)$}

\begin{lemma}\label{reco} Assume again that
$\rho$ is a faithful representation of
 $S_0\ltimes E(K)$. 

Then $\Lambda$ is a good subgroup of
$K^*$.
\end{lemma}
 
\begin{proof} Since any zero-rank subgroup
of $K^*$ is good,  we can assume that 
$\rk\,\Lambda\geq 1$. For $n\geq 3$, set
$E_n=K.T_n$. By
Lemma \ref{unirep}, $K$ and  $L$ 
have the same ground field $\F$ and,
$\rho(E_n)$ is a unipotent group.

{\it We claim that, for 
any $n\geq 3$,
there is a $L[\Lambda]$-submodule $W'\subset W(\rho)$ such that $W(\rho)/W'$  contains  a 
$\F[\Lambda]$-submodule 
$X$ isomorphic to $\F[\Lambda]/I_n(\Lambda)$.}

\smallskip\noindent
{\it Proof  for $\F=\Q$.} 
In that case $W'=\{0\}$ and  $X:=\log\rho(\Q[\Lambda].T_n)$
is a $\Q[\Lambda]$-submodule of $W(\rho)$ isomorphic to
$\Q[\Lambda]/I_n(\Lambda)$.

\smallskip\noindent
{\it Proof for $\F=\F_p$.} 
As a substitute for the $\log$,
set $\theta(a)=1-\rho(a)$ for  $a\in E_n$.  
Let $M\subset W(\rho)$ be the 
linear space generated by $\theta(E_n)$. 
Since 

\centerline{$\theta(a)\theta(b)=
\theta(a)+\theta(b)-\theta(a+b)$,}

\noindent 
$M$ is a nonunital algebra and $M^p=\{0\}$.
Since $\rho$ is injective, we have $\theta^{-1}(M^p)=\{0\}$. Thus there exists a unique integer
$m<p$ such that

\centerline{$Y:=\theta^{-1}(M^m)\neq \{0\}$ but
$\theta^{-1}(M^{m+1})= \{0\}$.}

Set $W'=M^{m+1}$. It follows from the previous formula that
$Y$ is a subgroup of $E_n$ and the induced map
$\overline\theta:Y\to W(\rho)/W'$ is additive.
Thus $\overline\theta$ is a homomorphism of 
$\F_p[\Lambda]$-modules.
Since $\Ker\, \overline\theta$ is trivial, 
$\overline\theta$ is injective.
Any cyclic $\F_p[\Lambda]$-sumodule  of $Y$ is isomorphic to $\F_p[\Lambda]/I_n(\Lambda)$, therefore $W(\rho)/W'$ contains 
a $\F_p[\Lambda]$-submodule 
$X$ isomorphic to $\F_p[\Lambda]/I_n(\Lambda)$.

\smallskip
{\it We claim now that, for any $n\geq 3$, 
$W(\rho)_{(\chi)}\neq 0$ for some
$\chi\in{\cal X}_n(\Lambda)$.}
 
 \noindent Let $Z$ be the $L$-vector space generated by $X$, and let
$Z=\oplus_{\chi\in\Omega_Z}\,Z_{(\chi)}$ be the decomposition of the
$L[\Lambda]$-module $Z$ into generalized weight spaces, where $\Omega_Z$ is the set of 
group homomorphisms $\chi:\Lambda\to L^*$ 
such that $Z_{(\chi)}\neq 0$.
For each $\chi\in\Omega_Z$, let
$I_{\chi}$ be the annihilator in $\F[\Lambda]$ of
$Z_{(\chi)}$. It follows that

\centerline{$I_n(\Lambda)=\cap_{\chi\in\Omega_Z}\,I_\chi$.}

Since $\Omega_Z$ is finite and $I_n(\Lambda)$ is a prime ideal, we have $I_n(\Lambda)=I_\chi$, for
some $\chi\in\Omega_Z$. Moreover the radical of
$I_{\chi}$ is $\Ker\hat\chi$, hence $I_n(\Lambda)=\Ker\hat\chi$. Thus by Lemma
\ref{invariant}, $\chi$ is a semi-algebraic character of degree $n$. Moreover we have
$W(\rho)_{(\chi)}\neq 0$, what proves the claim.

\smallskip\noindent
{\it End of the proof.} Assume otherwise, namely that 
$\Lambda$ is a bad subgroup of $K^*$.
Then by Lemma \ref{disjointness}, there is an infinite set $T$ of integers $n\geq 3$ such that
the family 
$({\cal X}_n(\Lambda ))_{n\in T}$ consists
of mutually disjoint sets. This would contradict
that the finite set of  generalized weights of  
$W(\rho)$ intersects each of them.
\end{proof}

\smallskip\noindent
{\it 8.4 Proof of the Nonlinearity Criterion}

\noindent
Let $S$ be a subgroup of $\SL(2,K)$. For
$\delta\in\P^1_K$,  let
$\Lambda_\delta\subset K^*$ be the subgroup
of all eigenvalues of elements $g\in S_\delta$.

\begin{NLcrit} Assume one of the following two hypotheses

(i) $\Lambda_\delta$ is  a bad subgroup of $K^*$ for some $\delta\in \P^1_K$, or

(ii) the function $\delta\in \P^1_K\to
d(\Lambda_\delta)$ is unbounded.

Then the group $\Aut_S\,\na^2_K$ is not linear, even over a ring. 
\end{NLcrit}

\begin{proof} By contraposition, we will assume that $\Aut_S\,\na^2_K$ is  linear
over a ring, and we will show that neither 
Assertion (i) nor Assertion 
(ii) holds.

By Lemma \ref{mixed},
$\Aut_S\,\na^2_K$ is the mixed product
$\Gamma:=S\ltimes\,
*_{\delta\in  \P^1_K} E_\delta(K)$.
Since $
\cap_{\delta\in\delta}\,S_\delta
\subset\{\pm1\}$ and
$T_3^{-\id}=-T_3$, it follows that
$\Core_\Gamma(\cap_{\delta\in\delta}\,S_\delta)$ is trivial. Hence by
Corollary \ref{criterion2}, $\Aut_S\,\na^2_K$ is  linear over a field. Let
$\rho: \Aut_S\,\na^2_K\to \GL(n,L)$ be an injective
homomorphism, for some algebraically closed field $L$ and some positive integer $n$.

Since $\rho$ provides a faithful representation of  $B_\delta\ltimes E_\delta(K)$, it follows from
Lemma \ref{reco} that $\Lambda_\delta$ is 
a good subgroup of $K^*$. Moreover
by Lemma \ref{estimate}, we have $d(\Lambda_\delta)\leq 2+n^2$. 

Therefore neither 
Assertion (i) nor Assertion 
(ii) holds.
\end{proof}

\smallskip\noindent
{\it 8.5 Proof of the Theorem A.1 of the introduction.}

\begin{MainA1} For any infinite field,
$\SAut_0\,\na^2_K$ is not linear, even over a field.
\end{MainA1}
 
\begin{proof} With the previous notations, we have 
$\Aut_0\,\na^2_K=\Aut_{\SL(2,K)}\,\na^2_K$,
and $\Lambda_\delta=K^*$, for any 
$\delta\in\P^1_K$. 
Therefore it is enough to  check that $K^*$  satifies one
of the two assertions of the 
Nonlinearity Criterion.

If $K$ is an infinite subfield of 
$\overline\F_p$, then $d(K^*)$ is infinite.
Otherwise $K$ contains $\Q$ or 
a transcendental element $t$.  For
the subgroup $\Lambda:=\langle 2\rangle$
in the first case or $\Lambda:=\langle t,t+1\rangle$ in the second case, it is clear that
$\rk\,\Lambda>\trdeg\,\Lambda$.
Hence $K^*$ itself is  a bad group.
\end{proof}

\noindent{\it 8.6 Comparison with Cornulier's Theorem}

\noindent Let   $G_{\mathrm{Cor}}$ be the
group of all automorphisms of $\na_\C^2$ of the form

\centerline{$(x,y)\mapsto (x+u, y+f(x))$,
for some $u\in \C$ and $f(t)\in \C[t]$.}

\noindent The group  $G_{\mathrm{Cor}}$ is locally nilpotent but not nilpotent \cite{Co}, 
hence

\begin{Cornu} Neither $G_{\mathrm{Cor}}$ nor 
$\Aut\,\na_\C^2$ is  linear over a field.
\end{Cornu}

\noindent 
Set $R:=\C[[t]]\oplus I$, where 
$I:=\C((t))/\C[[t]]$ is a  square-zero ideal.
Let $\Gamma$ be the subgroup of $\SL(2,R)$
generated by the matrices

\centerline{
$\begin{pmatrix} 1+t&0\\ 0&(1+t)^{-1}
\end{pmatrix}$ and
$\begin{pmatrix} 1&0\\ a&1
\end{pmatrix}$ when $a$ runs over $I$.}

\noindent The group $\Gamma$ is isomorphic to
 $G_{\mathrm{Cor}}\simeq \C\ltimes \C[x]$,
 where $\C$ acts by translation on $\C[x]$,
 hence $G_{\mathrm{Cor}}$ is linear over $R$.
Similarly, the subgroup $\SElem_0(K)$ is isomorphic to
$K^*\ltimes x^2K[x]$, where $K^*$ acts on
$x^2K[x]$ by multiplication. Since
$\SElem_0(K)$ embeds into
$\prod_{n\geq 2}\,(K^*\ltimes Kx^n)$, it is a subgroup of $\GL(2,K^\infty)$. 

Therefore both $G_{\mathrm{Cor}}$ and 
$\SElem_0(K)$ (for $K$ infinite) are linear over some rings but not over a field. This explains the motivation of Section 3.

\section{Two Linearity Criteria for $\Aut_S\,\na_K^2$}

For a subgroup $S$ of $\SL(2,K)$,
there are two linearity criteria for $\Aut_S\,\na_K^2$. The second one, stronger, is only proved  for a field $K$ of characteristic zero.

 Let $\Lambda\subset K^*$ be a subgroup. 
 For any $\F[\Lambda]$-module $M$
 and any $n\geq 1$, set 
 $M^{(n)}=(\rho_n)_*\,M$, where
 $\rho_n$  is the group 
 homomorphism 
 $x\in\Lambda\mapsto x^n\in\Lambda$.

\bigskip\noindent
{\it 9.1 The standard module for torsion-free 
good subgroups of $K^*$} 

\noindent  Assume now that
$\Lambda$ is a torsion-free good subgroup of
$K^*$.  By definition,
the {\it standard $\F[\Lambda]$-module} is
$K_1$, and it is denoted by $\St(\Lambda)$. 

Given a $K$-vector space $V$ and an integer
$n\geq 1$, it is clear that $V^{(n)}$ is
a direct sum of standard modules, and its multiplicity is the cardinal

\centerline{
$[V^{(n)}:\St(\Lambda)]=
[\Lambda:\Lambda^n]\, \dim_{K_1}V
=\dim_{K_n}\,V$.}

\bigskip\noindent
{\it 9.2 The First Linearity Criterion}

\noindent
Let $S$ be a subgroup of $\SL_S(2,K)$ and
let $\Lambda_\delta$ be the set of all
eigenvalues of $S_\delta$, for any $\delta\in\P^1_K$. Set $\SL_S(2,K[t]):=\{G\in \SL(2,K[t])\mid G(1)\in S\}$.

\begin{Lcrit}\label{Lcrit} Assume that
$\Lambda_\delta$ is a torsion-free good subgroup of $K^*$, for any $\delta\in\P^1_K$.
Then, for some field extension $L$ of $K$,
there is an embedding
 
\centerline{ $\Aut_S\,\na_K^2\subset 
\SL(2,L(t))$.}

Moreover if $\rk\,\Lambda_\delta\leq \aleph_0$ 
for any $\delta\in\P^1_K$, then we have

\centerline{$\Aut_S\,\na_K^2\simeq 
\SL_S(2,K[t])\subset \SL(2,K(t))$.}
\end{Lcrit}

\begin{proof} Set $M=\Sup\,\rk\,\Lambda_\delta$, where 
$\delta$ runs over $\P^1_K$.
There exists  a field extension $L\supset K$,
which satisfies one of the following two  hypotheses

$({\cal I}_1)$\hskip 1cm $[L:K]\geq M$ if
$M>\aleph_0$, or

$({\cal I}_2)$\hskip 1cm $L=K$  if 
$M\leq\aleph_0$

\noindent It follows from Lemmas \ref{mixed} and Lemma \ref{Tits} that 

\centerline{
$\Aut_S\,\na_K^2=S\ltimes *_{\delta\in\P^1_K}\,
E_\delta(K)$, and}

\centerline{$\SL_S(2,L[t])=
S\ltimes *_{\delta\in\P^1_L}\,U_\delta(zL[z])
\supset 
S\ltimes *_{\delta\in\P^1_K}\,U_\delta(zL[z])$.}

\noindent Let $\delta\in \P^1_K$. The 
$\F[\Lambda_\delta]$-modules $E_\delta(K)$ and
$U_\delta(zL[z])$ are copies of standard 
$\F[\Lambda]$-module. Since
$E_\delta(K)\simeq\oplus_{n\geq 3} K^{(n)}$, we have
 
\centerline{
$[E_\delta(K):\St(\Lambda)]=\sum_{n\geq 3}\,
[\Lambda:\Lambda^n][K:K_1]$.}

\noindent On the other hand, 
$U_\delta(zL[z])$ is 
isomorphic to $\aleph_0$ copies of $L^{(2)}$, 
therefore we have

\centerline{
$[U_\delta(zL[z]):\St(\Lambda)]=\aleph_0\,
[\Lambda:\Lambda^2][L:K][K:K_1]$.}

\noindent Hence
$({\cal I}_1)$ implies the existence of a 
$S_\delta$-equivariant embedding

\centerline{$\psi_\delta:
U_\delta(zL[z])\to E_\delta(K)$,}

\noindent and $({\cal I}_2)$ implies the existence
a $S_\delta$-equivariant isomorphism

\centerline{$\psi_\delta:U_\delta(zK[z])\to E_\delta(K)$.}

Therefore by Lemma \ref{universal},
$({\cal I}_1)$ implies the existence of a 
an embedding

\centerline{
$\Aut_S\,\na_K^2\subset 
S\ltimes *_{\delta\in\P^1_K}\,U_\delta(zL[z])
\subset \SL_S(2,L[t])\subset
\SL(2,L(t))$, }

\noindent and $({\cal I}_2)$ implies the existence
an isomorphism

\hskip0.5cm
$\Aut_S\,\na_K^2\simeq
S\ltimes *_{\delta\in\P^1_K}\,U_\delta(zK[z])
\simeq \SL_S(2,K[t])\subset\SL(2,L(t))$.
\end{proof}

\begin{MainD}\label{mainD} Let $R$ be a f.g. subring of $K$ and let 
$S\subset\SL(2,R)$. 

If $\rk\,\Lambda_\delta=\trdeg\,\Lambda_\delta$
for any $\delta\in\P^1_K$, then
$\Aut_S\,\na_K^2$ is linear  over 
$K(t)$.

Otherwise, $\Aut_S\,\na_K^2$ is  not linear,
even over a ring.
\end{MainD}

\begin{proof} The second assertion follows from
the Nonlinearity Criterion.

In order to prove the first one, assume now
that $\rk\,\Lambda_\delta=\trdeg\,\Lambda_\delta$
for any $\delta\in\P^1_K$.
Let $L$ be the field of fraction of
$R$ and set $F={\overline \F}\cap L$.
Since $R$ is a f.g. ring, $F$ is a finite
extension of $\F$. Therefore the set
$\mu\subset K$ of roots of unity in $F$ or in a quadratic extension of $F$ is finite.

Let $\bf m$ be a maximal ideal of $R$
such that $p:=\ch\,R/{\bf m}$ is
coprime to $\Card\,\mu$. Indeed if $\F=\Q$ 
the characteristic of $R/{\bf m}$ is
arbitrarily large, while in the opposite case, 
$\ch\,R/{\bf m}$ is
automatically coprime to $\Card\,\mu$.

Set $S'=\{g\in S\mid g\equiv\id \mod {\bf m}\}$.
For any $\delta\in P^1_K$ let $\Lambda'_\delta$ be the eigenvalues of $S'_\delta$.
Since $S'$ is residually $p$-group,
$\Lambda_\delta'$ is torsion-free.
Since $[S:S']<\infty$, we can assume that $S=S'$.

We claim that $\Lambda_\delta$ is f.g.
for any $\delta\in\P^1_K$. If $\delta$ is
defined over a field $F'$, where
$F'=F$ or $F'$ is
a quadratic extension of $F$, 
then $\Lambda_\delta\subset \overline{R}^*$, where
$\overline{R}$ is the integral closure of $R$ in
$F'$. Since $\overline{R}^*$ is f.g., so is
$\Lambda_\delta$. Otherwise,
$\Lambda_\delta$ is trivial.
Hence $\Lambda_\delta$ is a torsion-free
good subgroup of
$K^*$ for any $\delta\in\P^1_K$. 

The first Linearity Criterion implies that
$\Aut_S\,\na_K^2$ is linear  over $K(t)$.
\end{proof}

\bigskip\noindent
{\it 9.3 The standard modules for finite-torsion 
good subgroups of $K^*$} 

\noindent From now on, $K$ is a field of characteristic zero. 
Let $\Lambda$ be a good subgroup of $K^*$, such that $\Card\,\Lambda\cap \mu_\infty=n$ for some $n<\infty$.

The $\Q[\Lambda]$-module $\St_d(\Lambda):=K_d$,
where $d$ is a divisor of $n$, are called the
{\it standard $\Q[\Lambda]$-modules}.
By Baer Theorem \cite{Ba}, $\Lambda$ is isomorphic
to $\mu_n\times {\overline\Lambda}$,
where ${\overline\Lambda}=\Lambda/\mu_n$ is torsion-free. It follows that
$\St_d(\Lambda)\simeq 
\Q(\mu_{n/d} )\otimes \Q({\overline\Lambda}^d)$.

Given a $K$-vector space $V$ and an integer
$m\geq 1$, it is clear that $V^{(m)}$ is
a direct sum of the standard module
$\St_{\gcd(n,m)}(\Lambda)$, and its multiplicity is the cardinal

\centerline{
$[V^{(m)}:\St_{gcd(n,m)}(\Lambda)]=
[{\overline\Lambda}:{\overline\Lambda}^n]
\,\phi(n)/\phi({\gcd(n,m)})
\, \dim_{K_1}V$.}

\begin{lemma}\label{prepa} Let $S_0$ be a subgroup of $B$ such that $\chi_B(S_0)= \Lambda$ is a 
good subgroup
of $K^*$ such that $n:=\Card\,\Lambda\cap \mu_\infty$ is finite.
Let $l>n$ be a prime number
and let $L\supset K$ be a field such that
$[L:K]\geq \aleph_0 \rk(\Lambda)$. 

Then there is a $S_0$-equivariant embedding
$E(K)\subset E^{<2l}(L)$.
\end{lemma}

\begin{proof} Let $D$ be the set of divisors of $n$. The $\Q[\Lambda]$-module $E(K)$ and
$E^{<2l}(L)$ are  direct sums of standard modules, therefore we have

\centerline{$E(K)=\oplus_{d\in D}\,\St_{d}(\Lambda)^{m_d}$, and}

\centerline{$E^{<2l}(L)=\oplus_{d\in D}\,\St_{d}(\Lambda)^{n_d}$,}

\noindent where the multiplicities $m_d$ and 
$n_d$ are cardinals. Therefore it is enough to prove that $n_d\geq m_d$ for any $d\in D$.

 Since $E(K)=\oplus_{m\geq 3}\,K.T_m$, it is clear that 

\centerline{$m_d:=[E(K)(d):\St_{d}(\Lambda)]
\leq\aleph_0 \rk(\Lambda) [K:K_1]$.}

Similarly, we have 
$E^{<2l}(L)=\oplus_{3\leq m\leq 2l}\,L.T_m$.
Let $d\in D$. If  $d\geq 3$,
$L.T_d$ is a direct sum of standard modules
$\St_{d}(\Lambda)$ and it is clear that
$[L.T_d: \St_{d}(\Lambda)]\geq 
[L:K][K:K_1]\geq \aleph_0 \rk(\Lambda) [K:K_1]$,
and therefore $n_d\geq m_d$. 
Since $l$ is coprime to $n$, then  
$L.T_l$ is a direct sum of standard modules
$\St_{1}(\Lambda)$, and  we have
$n_1\geq m_1$. 

If $n$ is odd, the assertion is proved.
Otherwise, $L.T_{2l}$ is a direct sum of standard modules
$\St_{2}(\Lambda)$, and similarly we have
$n_2\geq m_2$. 
\end{proof}

\bigskip\noindent
{\it 9.4 The Second Linearity Criterion}

\noindent

\begin{Lcrit}\label{Lcrit2} Let $K$ be a field of characteristic zero. Assume that

(i) $\Lambda_\delta$ is a  good subgroup of $K^*$, for any $\delta\in\P^1_K$, and

(ii) the function $\delta\mapsto 
\Card\,\Lambda\cap \mu_\infty$ is bounded.

Then $\Aut_S\,\na_K^2$ is a linear group over
some field extension $L$ of $K$. 
\end{Lcrit}

\begin{proof} Let $L\supset K$ be a field such that $[L:K]\geq \aleph_0 M$, where
$M=\Sup \rk\,\Lambda_\delta$, and let
$Q\subset \P^1_K$ be a set of representatives of 
$\P^1_K/S_0$.  By Lemma \ref{prepa} there is
a $S_\delta$-embedding 
$\psi_\delta:E_\delta(K)\to E_\delta^{<2l}(L)$ where $l>\Max\, \Card\,\Lambda\cap \mu_\infty$ is a prime number.  Therefore we get some embeddings

\centerline{
$\Aut_S \na_K^2\simeq S\ltimes *_{\delta\in\P^1_K}\, E_\delta(K)
\subset S\ltimes *_{\delta\in\P^1_K}\, 
E^{<2l}_\delta(L)
\subset \Aut_0^{<2l} \na_L^2$.}

So, the stong version of Theorem C.2
implies  that
$\Aut_S \na_K^2$ is linear.
\end{proof}

\section{Examples of Linear or Nonlinear $\Aut_S\,\na_K^2$}

We provide three examples using the 
Linearity/Nonlinearity Criteria.

\bigskip\noindent
{\it 10.1 Example A, with $S=\SO(q)$
and $K$ infinite}

\begin{ExA} Let $q$ be a quadratic form on $K^2$ and $S=\SO(q)$. 

If $q$ is anisotropic, $\Aut_S\na^2_K$
is linear over a field extension of $K$.

Otherwise $\Aut_S\na^2_K$ is not linear, even over a ring.
\end{ExA}

\noindent As we will see, the proofs
for $\ch\,K=0$ and for $\ch\,K\neq 0$
are very different. In particular, the group $S=\SO(2,\R)$ has no subgroups of finite index and the proof for $K=\R$ cannot be reduced to the first
Linearity Criterion.

\begin{proof} If $q$ is isotropic or degenerate,
we have $\Lambda_\delta=K^*$ for some
$\delta\in\P^1_K$. The proof of Theorem A.1
shows that $K^*$ itself is bad. Hence by the Nonlinearity Criterion,
$\Aut_S\na^2_K$ is not linear, even over a ring. 

Assume now that $q$ is anisotropic.
Let $L\supset K$ be the  quadratic extension splitting $q$. Then
$\SO(q)$ is isomorphic to  
$S:=\{z\in L^*\mid N_{L/K}(z)=1\}$. Let
$S^{\infty}$ be the subgroup of all  
 $s\in S$ of order a power of $2$.

\smallskip\noindent
{\it 1. Proof for $\ch K=p$.}
We claim that $S^{\infty}$ is finite.
So we can assume that 
$\Card S^{\infty}\geq 4$. 
Since $\sqrt{-1}\in S$ and 
$S\cap K^*=\{\pm1\}$, it follows that
$L=K(\sqrt{-1})$. Therefore
$p\equiv -1\mod 4$, and 
$L=K.\F_{p^2}$. Let $s\in S^{\infty}$
of order $>2$. There is an integer $n\geq 1$
such that $s\in \F_{p^{2n}}$ where
$\F_{p^{n}}\subset K$. Since
$\F_{p^2}\not\subset K$, the integer $n$ is odd.
 Since $\Card \F_{p^{2n}}^*/\F_{p^2}^*$ is odd,
 $s$ belongs to $\F_{p^2}^*$. Hence
 $S^{\infty}\subset \F_{p^2}^*$ is finite.
 
 By Baer's theorem \cite{Ba}, we have 
$S=S^\infty\times S'$ for some subgroup $S'\subset S$. Since $S'_\delta=\{1\}$ for any $\delta\in\P^1_K$, the group
$\Aut_{S'}\,\na_K^2$ embeds into
$\SL(2,K(t))$ by the first Linearity Criterion. Since

\centerline{$[\Aut_S\,\na_K^2:\Aut_{S'}\,\na_K^2]
=[S:S']<\infty$,}

\noindent  the group $\Aut_{\SO(q)}\,\na_K^2$
is also linear over $K(t)$.

\smallskip\noindent
{\it 2. Proof for $\ch K=0$.}
Since $SO(q)_\delta=\{\pm1\}$ for any $\delta\in\P^1_K$, the group $\Aut_{\SO(q)}$ is linear
over a field extension of $K$ by the second
Linearity Criterion.
\end{proof}

 \bigskip\noindent
{\it 10.2 A preparatory lemma for the example B}

\noindent
We did not found a reference for the 
next well-known result. The given proof that
$l(\gamma_a)$ is arbitrarily large is due to
Y. Benoist. Also it is implicit in \cite{Be} that
$l(\gamma_a)$ is not constant, as pointed out by
I. Irmer. A third proof is based on the fact that 
any loxodromic representation $\Pi_g\to \PSL(2,\R)$
can be deformed to the trivial representation inside
the character variety, i.e, the space of semi-simple
complex representations $\Pi_g\to \PSL(2,\C)$.
All these proofs involves some geometric arguments.

\begin{lemma}\label{Baire} There are
cocompact  lattices 
$S\subset \SL(2,\R)$ such that
$\Tr\,\gamma$ is a transcendental number
for any infinite order element
$\gamma\in S$.
\end{lemma}

\begin{proof} Let $g\geq 2$. Let
$\Sigma_g$ be the oriented Riemann surface of genus $g$, let $T(\Sigma_g)$ be its Teichm{\"u}ller space and let
$\Pi_g$ be the group presented by

\centerline{
$\langle\alpha_1,\beta_1\dots\alpha_g,\beta_g\mid
\prod_{1\leq k\leq g}\,(\alpha_k,\beta_k)=1
\rangle$.}

Let $\gamma$ be a conjugacy class in 
$\Pi_g$. Any point $a\in T(\Sigma_g)$
determines a group homorphism
$\rho_a: \Pi_g\to\pi_1(\Sigma_g)$ and
an hyperbolic metric $g_a$ on $\Sigma_g$, up to some equivalence \cite{IT}, ch.5. 
Hence $\rho_a(\gamma)$  is represented by a  unique closed $g_a$-geodesic
$\gamma_a:S^1\to\Sigma_g$, i.e. a geodesic relative to the metric $g_a$. 

In elementary terms,
$\gamma$ is represented, modulo 
$\PSL(2,\R)$-conjugacy,  by a hyperbolic element
$h_a\in\PSL(2,\R)$. The complete geodesic 
$\Gamma\subset\HP$ whose the extreme points in 
$\partial\HP$ are the fixed points of $h_a$,
is the locus of minima for the function 
$z\in\HP\mapsto d_\HP(z,h_a.z)$. The  
$g_a$-geodesic $\gamma_a$ is, up to reparametrization, the image in $\Sigma_g$
of any segment $[z,h_a.z]$ of $\Gamma$.

We claim that the length function $a\mapsto l(\gamma_a)$ is not constant. Let us pick another conjugacy class $\delta$ in $\Pi_g$ such that
$\rho_a(\delta)$ is represented by
a simple geodesic $\delta_a$ which meets
$\gamma_a$ transversally (since
$T(\Sigma_g)$ is connected, this condition is independent of $a$).
By a corollary of the collar theorem,

\centerline{$\sh\,{l(\gamma_a)\over 2}
\,\sh \, {l(\delta_a)\over 2}>1$,}

\noindent see
4.1.2 in \cite {B}. 
Since $l(\delta_a)$ can be arbitrarily small,
$l(\gamma_a)$ is arbitrarily large.

Each $a\in T(\Sigma_g)$ determines an homomorphism 
$\pi_1(\Sigma_g)\to \PSL(2,\R)$ up
to some equivalence, which induces
a homomorphism

\centerline{${\tilde\rho}_a:\Pi_g\to\PSL(2,\R)$}

\noindent
 defined up to conjugacy by $\PSL(2,\R)$.

For $g\in \PSL(2,\R)$,  
$\Tr\, g^2$ is well-defined. We have
 
\centerline{ $\Tr\,{\tilde\rho}_a(\gamma)^2=
2\ch\,l(\gamma_a)$. }

Let ${\cal I}$ be the set of irreducible
polynomials in $\Q[t]$. For $P\in {\cal I}$,
set 

\centerline{$\Omega(\gamma,P)=
\{a\in T(\Sigma_g) 
\mid P(\Tr\,{\tilde\rho}_a(\gamma)^2)\neq 0\}$.}

\noindent
 Since the function $a\mapsto P(\Tr\,{\tilde\rho}_a(\gamma)^2)$ is nonconstant and analytic,
$\Omega(\gamma,P)$ is a dense open subset of
$T(\Sigma_g)$. By the Baire Theorem

\centerline
{$\Omega:=\cap_{\gamma\neq 1, P\in{\cal I}}\,\Omega(\gamma,P)$}

\noindent is dense.
For any $a\in\Omega$, the lattice

\centerline{$S:=\{s\in \SL(2,\R)\mid s\mod {\pm 1}
\,\textnormal{belongs to}\,
{\tilde\rho}_a(\Pi_g)\}$}

\noindent satisfies the 
required condition. (Indeed 
$S\simeq \Pi_g\times \{\pm1\}$ by \cite {Pe}.)
\end{proof}

 \bigskip\noindent
{\it 10.3 Example B, where $S$ is a lattice}

\begin{ExB} 
 For some cocompact lattices 
$S\subset \SL(2,\R)$,  the group
$\Aut_S\,\na^2_\C$ is linear over $\C$.

For any lattice $S\subset \SL(2,\C)$, $\Aut_S\,\na^2_\C$ is not linear, even over a ring.

\end{ExB}

\begin{proof} Let $S\subset \SL(2,\R)$ 
be a cocompact lattice as in Lemma \ref{Baire}.
Then for any
$\delta\in\P^1_\C$, it is clear that
$\Lambda_\delta$ has rank one and contains a transcendental element, or it is finite.
Hence $\Aut_S\,\na^2_\C$ is linear over $\C$
by Theorem D.

Let  $S$ be a lattice of $\SL(2,\C)$, let 
$g\in S$ be of infinite order
and let $\delta\in\P^1_\C$ be a fixed point  of $g$. By the Garland-Raghunathan rigidity theorem 0.11 of \cite{GR},
the eigenvalues of $g$ are algebraic numbers.
Since
$\rk\,\Lambda_\delta>\trdeg\,\Lambda_\delta$,
Theorem D implies that 
$\Aut_S\,\na^2_\C$ is not linear, even over a ring.
\end{proof}

 \bigskip\noindent
{\it 10.4 A preparatory lemma for example C}

\begin{lemma}\label{rank1} Let $R$ be a prime
normal ring with fraction field $k$  and let
$m\geq 1$. Let $B$ be the integral closure of $R[t_1,\dots,t_m]$ in some quadratic extension
$L_1\subset k((t_1,\dots,t_m))$  of 
$k(t_1,\dots,t_m)$.

Then  $B^*/R^*$ is isomorphic to $\{1\}$ or $\Z$.
\end{lemma}

\begin{proof} 
Since $R$ is normal, we have $B\cap k=R$, hence the
map $B^*/R^*\to (k\otimes B)^*/k^*$ is one to one. So we can assume that $R=k$.
Set $C=\Spec \, B$. There is a unique normal
compactification $\overline C$ of $C$ such that
the finite map
$C\to\na_k^m=\Spec\,k[t_1,\dots t_m]$ extends to a finite map
$\pi:\overline C\to \P_k^m$. 

Set $Z:=\pi^{-1}(\P_k^{m-1})$, where
$\P_k^{m-1}:=\P_k^m\setminus\na_k^m$.
For any irreducible divisor $D$ in
$\overline C$, let $v_D$ be the corresponding valuation.

If $Z$ is irreducible, then $v_Z(f)\leq 0$
for any $f\in B$. Hence  $v_Z(f)=0$
for any $f\in B^*$, and therefore $B^*=k^*$.

Otherwise, $Z$ is the union of two divisors
$Z_1$ and $Z_2$. For any 
$f\in B^*\setminus k^*$, either
$v_{Z_1}(f)<0$ or $v_{Z_2}(f)<0$.
Hence the homomorphism
$f\in B^*\mapsto (v_{Z_1}(f),v_{Z_2}(f))\in\Z^2$ 
embeds $B^*/k^*$ in a free $\Z$-module of rank 
$\leq 1$.
\end{proof}

\bigskip\noindent
{\it 10.5 Example C, where $\rk\,\Lambda_\delta$
is not constant}

\noindent
Let $A$ be a torsion-free additive group of any rank, let $d, m>0$ be  integers with
$d$ square-free
and let ${\cal O}$ be the ring of integers
of $k:=\Q(\sqrt{d})$. Set 
$K=k(A)((t_1,\dots,t_m))$ 
where $k(A)$ is the field of fractions of $k[A]$, and

\centerline{$S:=
\SL(2,{\cal O}[A][t_1,\dots,t_m])$.} 

\noindent The Example $C$ of the introduction is the case  $A=\Z$ and $m=1$. 

\begin{ExC} If $d<0$, then $\Aut_S\,\na_K^2$ is linear over a field extension of $K$.

Otherwise, $\Aut_S\,\na_K^2$ is not linear, even over a 
ring.
\end{ExC}

\begin{proof} Set $L_0:=k(A)(t_1,\dots,t_m)$.

\smallskip\noindent
{\it Proof if $d>0$.}
Then  $\Q(\sqrt{d})$
is a real field, we have
$\rk\,{\cal O^*}=1>\trdeg\,{\cal O^*}=0$. For $\delta\in \P^1_{L_0}$, the group
$\Lambda_\delta={\cal O^*}\times A$ is  a
bad subgroup of
$K^*$. So, by the Nonlinearity Criterion,
$\Aut_S\,\na_K^2$ is not linear, even over a 
ring.

\smallskip\noindent
{\it Proof if $d<0$.}
Let $\delta\in \P^1_{L}$.

If  $\delta$ belongs to $\P^1_{L_0}$,
we have $\Lambda_\delta={\cal O}^*\times A$.
Since $\Q(\sqrt{d})$ is an imaginary field, 
$\Lambda_\delta={\cal O}^*\times A$ 
is  a good subgroup of $K^*$ and 
$\Card\,\Lambda_\delta\cap\mu_\infty\leq 6$.

Assume now that $\delta$ belongs to 
$\P^1_{L_1}\setminus\P^1_{L_0}$, where
$L_1$ is a quadratic extension of
$L_0$. Let $B$ be the algebraic closure of
${\cal O}[A][t_1,\dots,t_m]$ in $L_1$
and let   $N:=\{z\in L_1\mid N_{L_1/L_0}(z)=1\}$
be the norm group. 
It is clear that

\centerline{$\Lambda_\delta=N\cap B^*$.}

By Lemma \ref{rank1}, we have
$\Lambda_\delta=\{\pm 1\}$ or
$\Lambda_\delta=\{\pm 1\}\times \Z$. 
Since ${\cal O}$ is algebraically closed in $L$,
it should be noted that when $\rk\,\Lambda_\delta=1$,
we also have $\trdeg\,\Lambda_\delta=1$. Thus
$\Lambda_\delta$ is  a good subgroup of
$K^*$ and $\Card\,\Lambda_\delta\cap\mu_\infty= 2$.

Otherwise, we have 
$\Lambda_\delta=\{\pm 1\}$ and the same conclusion holds.

Therefore, by the second Linearity Criterion,
$\Aut_S\,\na_K^2$ is linear over some field
extension of $K$.
\end{proof}

\section{Nonlinearity of Finite-Codimensional\\ Subgroups of $\Aut\,\na_K^3$}

Theorem A.2 shows that $\Aut\,\na_K^2$ contains some finite-codimensional subgroups,
which are linear as abstract groups. However, this result does not extend to $\na_K^n$, for 
$n\geq 3$, as it will be shown in this section. 
 
 For our purpose, the case $n=3$ is enough. 
Unlike in the introduction, {\it it will be convenient to use the  coordinates $(z,x,y)$ for $\na_K^3$}. Let
$\TAut\,\na_K^3$ be the the subgroup of tame automorphisms of $\na_K^3$, see 11.5 for the definition.  By the famous result of 
Shestakov and Urmibaev \cite{SU}, 
$\TAut\,\na_K^3$ is a {\it proper} subgroup of 
$\Aut\,\na_K^3$. 

Let ${\bf m}$ be a
 finite-codimensional ideal in $K[z,x,y]$. Let
 $\Aut_{\bf m}\,\na_K^3$ be the group of all polynomial automorphisms $\phi$ of the form
 
 \centerline{$(z,x,y)\mapsto 
 (z+f, x+g, y+h)$,}
 
 \noindent where $f,\,h$ and $g$ belongs to ${\bf m}$. Set 

\centerline{$\TAut_{\bf m}\,\na_K^3=
\TAut\,\na_K^3 \cap\Aut_{\bf m}\,\na_K^3$.}
 
The nonlinearity result for $n=3$, valid even if $K$ is finite, is unrelated with the existence of wild automorphisms
in $\na^3_K$, as shown by
 
 \begin{MainB}  For any finite-codimensional ideal ${\bf m}$
 of $K[z,x,y]$, the groups $\Aut_{\bf m}\,\na_K^3$ and
 $\TAut_{\bf m}\,\na_K^3$ are not linear, even over a ring.
  \end{MainB}
 
 The proof  uses the folklore embedding $\Phi$, likely known by
Nagata \cite{Na}, and used in \cite{SU}.
The simplest obstruction for  the linearity
is due to a
nonnilpotent locally nilpotent subgroup.
In characteristic zero, the proof is easy, and it
follows  the line of \cite{Co} together with Corollary 2. In characteristic $p$, the proof 
involves  the strange formula of Lemma \ref{formula}.

\bigskip\noindent
{\it 11.1 Nilpotency class of some p-groups}

\noindent 
The {\it nilpotency class} of a nilpotent group is the lenght of its ascending central series.
Let $p$ be a prime integer, and let $E$ be an
elementary $p$-group of rank $r$. Note that  $E$ acts by translation on $\F_p[E]$ and set $G(r)=E\ltimes \F_p[E]$.  

Set $E=D_1\times\dots\times D_r$, where
each $D_i$ has rank $1$.
For each $i$, the
socle filtration of the
$D_i$-module $\F_p[D_i]$ has length $p$
and we have $\F_p[E]=
\F_p[D_1]\otimes\dots\otimes \F_p[D_n]$.
Hence the socle filtration of the $E$-module $\F_p[E]$ has lenght $1+(p-1)r$. It follows that

\begin{lemma}\label{G(r)}
The nilpotency class of the nilpotent group $G(r)$ is $1+(p-1)r$.
\end{lemma}

Let $M$ be a cyclic $\F_p[E]$-module generated by some $f\in M$. 

\begin{lemma}\label{free} If $\sum_{u\in E} \, u.f\neq 0$, then  the $\F_p[E]$-module $M$ is free of rank one.
\end{lemma}

\begin{proof}Set $N=\sum_{u\in E}\,e^u$,
where $(e^u)_{u\in E}$ is the usual basis of
$\F_p[E]$. 
Note that $\F_p.N= H^0(E,\F_p[E])$, hence any nonzero ideal of $\F_p[E]$ contains $N$. Since $N.f\neq 0$, $M$ is freely generated by $f$.
\end{proof}

\bigskip\noindent
{\it 11.2 A formula}

\begin{lemma}\label{formula} Let $A$ be a 
commutative $\F_p$-algebra and let
$E\subset A$ be a linear subspace of 
dimension $r$. We have

 \centerline{$\sum_{u\in E} \,u^{p^r-1}=
\prod_{u\in E\setminus\{0\}}\, u$.}
\end{lemma}

\begin{proof} It is enough to prove the claim for
$A=\F_p[x_1,\dots,x_r]$ and 
$E=\oplus_{i} \F_p.x_i$.
Set $P(x_1,\dots,x_n):=\sum_{u\in E} \,u^{p^r-1}$,
set $H=\F_p.x_2\oplus \F_p.x_3\dots\oplus\F_p.x_r$ and for $v\in H$, set
$Q_v:=\sum_{\lambda\in \F_p}\,
(\lambda x_1+ v)^{p^r-1}$.
For any integer $n\geq 0$, we have
$\sum_{\lambda\in \F_p}\,\lambda^n=0$
except if $n$ is a positive multiple of $p-1$. Hence

\centerline{
$Q_v=\sum_{n>0} \,c_n
 x_1^{n(p-1)} v^{p^r-1-n(p-1)}$,}

\noindent for some $c_n\in \F_p$ and the polynomial
$Q_v$ is divisible by $x_1^{p-1}$.
Since 

\centerline
{$P(x_1,\dots,x_n)=\sum_{v\in H}\,Q_v$,}

\noindent the polynomial $P(x_1,\dots,x_n)$ is divisible by $x_1^{p-1}$. By $\GL(r,\F_p)$-invariance, $P(x_1,\dots,x_n)$ is divisible by $u^{p-1}$ for all $u\in E\setminus\{0\}$.
Hence it is divisible by
$\prod_{u\in E\setminus\{0\}}\, u$. Since both polynomials have
degree $p^r-1$, it follows that

\centerline{$P(x_1,\dots,x_n)=c
\prod_{u\in E\setminus\{ 0\}}\, u$,}

\noindent for some $c\in \F_p$.

As it is a universal constant, we can compute
$c$ for $A=E=\F_{p^r}$. Since 

\centerline{
$\sum_{\lambda\in \F_p^r}\,\lambda^{p^r-1}=-1$, and
$\prod_{\lambda\in \F_p^r}\,\lambda=-1$,}

\noindent it follows that $c=1$.
\end{proof}

\bigskip\noindent
{\it 11.3 The locally nilpotent group $G(I)$}

\noindent Let $\F$ be a prime field. For any ideal $I$ of a commutative $\F$-algebra
$A$,  let us consider the semi-direct product 
$G(I):=I\ltimes I[t]$, where $I$ acts by
translation on the space $I[t]$
of polynomials with coefficients in $I$.

Let $E\subset I$ be an  additive subgroup, let $f(t)\in I[t]\setminus\{ 0\}$ and let $M$ be the additive subgoup generated 
 by all polynomials $f(t+u)$ when $u$ runs over 
 $E$. The group $E\ltimes M$, which  is a subgroup of $G(I)$, is obvioulsly nilpotent.

\begin{lemma}\label{index}
Assume that the algebra $A$ is prime.

(i) If $\F=\Q$,  the nilpotency class of $E\ltimes M$ is  $1+\deg\,f$.

(ii) Assume that $\F=\F_p$, that 
$\dim_{\F_p}\,E=r$ and that $f(t)=a x^{p^r-1}$ for some $a\in I\setminus\{ 0\}$. Then the group 
$E\ltimes M$ has nilpotency class 
$1+r(p-1)$.

(iii) If $\dim_{\F_p}\,I=\infty$, 
the group $G(I)$ is locally nilpotent but not nilpotent.
\end{lemma}

\begin{proof} Assertion (i) is obvious, and Assertion
(iii) is a consequence of Assertions (i) and (ii).
We will  prove  Assertion (ii), for which $\F=\F_p$.

Set  $g(t)=\sum_{u\in E}\,f(t+u)$. We have 
$g(0)=\sum_{u\in E} \,a u^{p^r-1}=
a\prod_{u\in E\setminus\{0\}}\, u$  by
Lemma \ref{formula}.
Since
$g(0)\neq 0$, it follows from Lemma \ref{free} that
the $\F_p[E]$-module $M$ is free of rank one,
and $E\ltimes M$ is isomorphic to $G(r)$. Thus
its nilpotency class is $1+r(p-1)$ 
by Lemma \ref{G(r)}.
\end{proof}

\bigskip\noindent
{\it 11.4 The amalgamated product 
$\Aff(2,I)*_{B_{Aff}(I)}\Elem(I)$}

\noindent From now on, let $I$ be a proper
nonzero ideal in $K[z]$.

\begin{lemma}\label{ElemI}
The group $\Elem(I)$ is not linear over a field.
\end{lemma}

\begin{proof} 
Since $I$ is a proper ideal, 
$\Elem(I)$  is the group of all automorphisms

\centerline{$\phi:(x,y)\mapsto (x+u, y+f(x))$,}

\noindent for some $u\in I$ and $f\in I[x]$. It follows that
$\Elem(I)$ is isomorphic to the group 
$G(I)$, and therefore $\Elem(I)$  contains subgroups
of abitrarily large nilpotency class by Lemma
\ref{index}. Hence $\Elem(I)$ is not linear over a field.
\end{proof}

\begin{lemma}\label{HI} The group
$\Aff(2,I)*_{B_{Aff}(I)}\Elem(I)$ is not
linear, even over a ring.
\end{lemma}

\begin{proof} By Lemma \ref{ElemI}, the group
$\Elem(I)$ is not linear over a field. Therefore
by Corollary 1, it is enough to show that
the amalgamated product $\Gamma:=\Aff(2,I)*_{B_{Aff}(I)}\Elem(I)$ satisfies

\centerline{$\Core_\Gamma(B_{Aff}(I))=\{1\}$.}

In order to do so, we first define two 
specific automorphisms $\gamma$ and $\phi$ as follows.
Let $r\in I\setminus\{0\}$ and
let $n\geq 3$ be an integer coprime to 
$\ch\,K$. Let $\gamma\in \Aff(2,I)$ be the
linear map $(x,y)\mapsto (x+ry,y)$ and let
$\phi\in \Elem(I)$ be the polynomial
automorphisms 
$(x,y)\mapsto (x,y+rx^n)$.

Let $g$ be an arbitrary element of  $B_{Aff}(I)\setminus \{1\}$.
By definition,
$g$ is an affine map
$(x,y)\mapsto(x+u,y+v+wx)$ for some $u,\,v.\,w\in I$.

If $w\neq 0$,  the linear part of
$g^\gamma$ is not lower triangular, therefore
$g^\gamma$ is not in $B_{Aff}(I)$.  If $w=0$ but $u\neq 0$, then the leading term $g^\phi$,
which is  
$(x,y)\mapsto (0, nru x^{n-1})$, has degree $\geq 2$. Therefore 
$g^\phi$ is not in $B_{Aff}(I)$. Last if $u=w=0$,
then $v$ is not equal to zero. It follows
that the leading term of $g^{\gamma\phi}$, which is $(x,y)\mapsto (0, nr^2v x^{n-1})$, 
has degree $\geq 2$.
Therefore 
$g^{\gamma\phi}$ is not in $B_{Aff}(I)$.

Hence, for any $g\in B_{Aff}(I)\setminus\{ 1\}$
at least one of the three elements
$g^\gamma$, $g^\phi$ or $g^{\gamma\phi}$ is not in $B_{Aff}(I)$. Therefore 
$\Core_\Gamma(B_{Aff}(I))$ is trivial.
\end{proof}

\bigskip\noindent
{\it 11.5 Proof of Theorem B}

\noindent The group $\TAut\,\na_K^3\subset \Aut\,\na_K^3$ of {\it
tame automorphisms} of $\na_K^3$ is

\centerline{
$\TAut\,\na_K^3=\langle\Aff(3,K),T(3,K)
\rangle$,}

\noindent where $\Aff(3,K)$ is the group of affine automorphisms of $\na_K^3$ and 
$T(3,K)$ is the group
of all triangular automorphisms

\centerline{$(z,x,y)\mapsto (z, x+f(z), y+g(z,x))$,}

\noindent where $f$ and $g$ are polynomials.
Note that $\Aut\,\na_{K[z]}^2=\Aut_{K[z]}\, K[z,x,y]$ is 
obviously the subgroup of 
$\Aut\,\na_K^3=\Aut_K\, K[z,x,y]$ of all automorphisms of the form

\centerline{$(z,x,y)\mapsto (z, f(z,x,y), g(z,x,z))$,}

\noindent where $f$ and $g$ are polynomials.

It is easy to see that
the groups $\Elem(K[z])$
and $\Aff(2,K[z])$ are subgroups of
$\TAut\,\na_K^3$. Therefore van
der Kulk's Theorem for the field $K(z)$
and Lemma \ref{subamal} imply that the homomorphism

\centerline
{ $\Phi:\Aff(2,K[z])*_{B(K[z]}\,\Elem(K[z])
\to \Aut\,\na_{K[z]}^2\cap 
\TAut\,\na_K^3$
}

\noindent is an embedding,  likely known by Nagata.
The hard and  beautiful result of \cite{SU} states that $\Phi$ is onto, a result which is not needed here. 

Now, we prove Theorem B.

\begin{proof} Without loss of generality, we can assume that the  ideal ${\bf m}$ is also a proper ideal.  Hence the ideal   $I:={\bf m}\cap K[z]$ is nonzero and proper.

The previously defined
morphism $\Phi$ induces an embedding 

\centerline{
$\Aff(2,I)*_{B_{Aff}(I)}\Elem(I)\to
\Aut_{\bf m}\,\na_K^3$.} 
 
\noindent
Hence by Lemma \ref{HI}, the group $\TAut_{\bf m} \,\na_K^3$
is not linear, even over a ring.
\end{proof}

\bigskip\noindent
{\it Acknowledgements.} J.P. Furter and R. Boutonnet informed us that they also found a f.g. subgroup of $\Aut\, \na_Q^2$
which is not linear over a field \cite{BF}. 

We  heartily thank S. Cantat, Y. Cornulier, 
I. Irmer, S. Lamy, J.-C. Sikorav, I. Soroko, and the referee for interesting comments. {\it Special thanks are due} to
Y. Benoist for his help in the proof of Lemma \ref{Baire}, 
T. Delzant for bringing my attention to
\cite{DS} and E. Zelmanov for an inspiring talk.

We also thank the hospitality of the
International Center for Mathematics
at SUSTech, where this work was partly done.


\begin{thebibliography}{1}
 
 
 
 \bibitem{Ba} R. Baer, Erweiterung von Gruppen und ihren Isomorphismen, Mathematische Zeitschrift 38
 (1934) 375-416.
 
 
\bibitem{BL} H. Bass and A. Lubotzky, Automorphisms of groups and of schemes of finite type, Israel J. Math. 44 (1983)1-22.



\bibitem{BPZ} O. Bezushchak, V. Petravchuk and E. Zelmanov, Automorphisms and Derivations of  Affine
Commutative  and PI-Algebras,
 Preprint (2022). 
\url{ https://arxiv.org/pdf/2211.10781.pdf}
 

 
\bibitem{BT} A. Borel and J. Tits. Homomorphismes ”abstraits” de groupes alg\'ebriques simples. Ann.
Math.128 (1973) 499-571.

 \bibitem{BF} R. Boutonnet and J.Ph. Furter,
 private communication.
 
\bibitem{Be} M. Bestvina, K. Bromberg, K. Fujiwara, J. Souto, Shearing coordinates and convexity of length functions on Teichmüller space, 
American Journal of Mathematics 135 (2013) 1449-1476.

\bibitem{B} P. Buser, {\it Geometry and Spectra of Compact Riemann Surfaces}, Birkh{\"a}user,
Progress in Math. 106 (1992).

\bibitem{C}
S. Cantat, Sur les groupes de transformations birationnelles des surfaces, Ann. of Math. 174 (2011) 299-340.

\bibitem{Co}
Y. Cornulier, Nonlinearity of some subgroups of the 
planar Cremona group, Preprint (2017).
\url{http://arxiv.org/abs/1701.00275v1}


\bibitem{D} M. Demazure, Sous-groupes alg\'ebriques de rang maximum du groupe de Cremona,
Ann. sc. de l'Ecole Normale Sup\'erieure, 4 (1970) 507-588.

\bibitem{DS} C. Drutu and M. Sapir,
Non-linear residually finite groups,
Journal of Algebra 284 (2005) 174–178.


\bibitem{FN}
W. Fenchel and J. Nielsen, {\it Discontinuous groups of isometries in the hyperbolic plane}, De Gruyter Studies
in Mathematics 29(2003).

\bibitem{FP} E. Formanek and C. Procesi, The automorphism group of a free group is not linear. J. Algebra 149 (1992) 494-499.

\bibitem{GR} H. Garland and M.S. Raghunathan,
Fundamental domains for lattices in
($\R$)-rank $1$ semi-simple Lie groups.
Ann. of Math. 2 (1970) 279-326.

\bibitem{IT} Y. Imagushi and M. Tanigushi,
{\it An Introduction to Teichm{\"u}ller spaces}
Springer-Verlag (1992).



\bibitem{L} S. Lamy, L'alternative de Tits pour $\Aut[\C^2 ]$, Journal of Algebra  239 (2001) 413-437.

\bibitem{MKS} W. Magnus, A. Karass and D. Solitar
{\it Combinatorial Group Theory}, Dover (1966).

\bibitem{N} H. Nagao, On $GL(2,K[x])$, 
J. Poly. Osaka Univ., 10 (1959) 117-121.

\bibitem{Na} M. Nagata, {\it On Automorphism Group of 
$k[x, y]$.} Kyoto University Lectures in Mathematics 5 (1972).

\bibitem{Pe} H. Petersson, Zur analytischen Theorie der Grenzkreisgruppen, III. Math. Ann. 115
(1938) 518-572.


\bibitem{P} V. Popov, Embeddings of Groups Aut($F_n$)
into Automorphism Groups of Algebraic Varieties,
Preprint (2022).
http://arxiv.org/abs/2106.02072v2

\bibitem{S83} J.P. Serre {\it Arbres, amalgames, $SL_2$} Ast\'erisque 49 (1977) or its translation  {\it Trees},
Springer-Verlag (1980).


\bibitem{S00} J.P. Serre, Le groupe de Cremona et ses sous-groupes finis, S\'eminaire Bourbaki 1000,
Ast\'erisque 332 (2010) 75-100.

\bibitem{SU} I. Shestakov and U. Urmibaev,
The tame and the wild automorphisms of polynomial rings in three variables, J. Am. Math. Soc. 17 (2004) 197-227.

\bibitem{T72} J. Tits, Free subgroups in linear groups. J. Algebra, 172 (1972) 250-270.

\bibitem{T82} J. Tits. R\'esum\'e de cours, Annuaire du 
Coll\`ege de France, 82e ann\'ee
(1981-1982) 91-105.

\bibitem{T87} J. Tits, Uniqueness and Presentation of Kac-Moody Groups over Fields, J. of Algebra 105 (1987) 542-573. 


\bibitem{T89} J. Tits,
Groupes associ\'es aux alg\`ebres de Kac-Moody,
Bourbaki Seminar 700. Ast\'erisque,  177-178 (1989) 7-31.

\bibitem{vdK} W. van der Kulk. On polynomial rings in two variables. Nieuw Arch. Wiskd.114 (1953) 33-41.

\bibitem{We} B.A.F. Wehrfritz, Generalized free products of linear groups, Proc. London Math. Soc. (27 (1973) 402-424.

 
\end{thebibliography}
\end{document}